\documentclass{article}
\usepackage{psfrag}
\usepackage[parfill]{parskip}
\usepackage{float, graphicx}
\usepackage[]{epsfig}
\usepackage{amsmath, amsthm, amssymb}
\usepackage{epsfig}
\usepackage{verbatim}
\usepackage{multicol}
\usepackage{url}
\usepackage{latexsym}
\usepackage{mathrsfs}
\usepackage[colorlinks, bookmarks=true]{hyperref}
\usepackage{graphicx}
\usepackage{amsmath}
\usepackage{enumerate}
\usepackage[normalem]{ulem}
\usepackage{bm}
\usepackage{dsfont}
\usepackage{stmaryrd}
\usepackage{cite}

\usepackage{color}

\makeatletter
\def\thm@space@setup{%
  \thm@preskip=\parskip \thm@postskip=0pt
}
\makeatother

\numberwithin{equation}{section}

\renewcommand{\cal}{\mathcal}
\newcommand\cA{{\mathcal A}}

\newcommand{\cE}{{\cal E}}
\newcommand{\cG}{{\cal G}}

\newcommand{\cN}{{\cal N}}

\newcommand\cZ{{\mathcal Z}}
\newcommand\cY{{\mathcal Y}}

\newcommand{\fa}{{\frak a}}
\newcommand{\fb}{{\frak b}}
\newcommand{\fc}{{\frak c}}
\newcommand{\fd}{{\frak d}}

\newcommand{\bma}{{\bm{a}}}
\newcommand{\bmb}{{\bm{b}}}

\newcommand{\bme}{\bm{e}}

\newcommand{\bmq}{{\bm q}}

\newcommand{\bmu}{{\bm{u}}}

\newcommand{\rQ}{{\rm Q}}
\newcommand{\rd}{{\rm d}}
\newcommand{\rU}{{\rm U}}
\newcommand{\ri}{\mathrm{i}}

\newcommand{\bC}{{\mathbb C}}
\newcommand{\bE}{\mathbb{E}}

\newcommand{\bP}{\mathbb{P}}

\newcommand{\bR}{{\mathbb R}}

\newcommand{\bT}{\mathbb T}

\newcommand{\al}{\alpha}

\newcommand{\e}{{\varepsilon}}

\newcommand{\la}{\lambda}

\DeclareMathOperator{\Tr}{Tr}

\DeclareMathOperator{\supp}{supp}

\DeclareMathOperator{\diag}{diag}

\renewcommand{\Re}{\mathop{\mathrm{Re}}}
\renewcommand{\Im}{\mathop{\mathrm{Im}}}

\newcommand{\deq}{\mathrel{\mathop:}=} 
 
\renewcommand{\leq}{\leqslant}
\renewcommand{\geq}{\geqslant}

\newcommand{\nc}{\normalcolor}

\newcommand{\del}{\partial}

\newcommand{\rn}[1]{
       \romannumeral#1
}

\newcommand{\qq}[1]{[\![{#1}]\!]}

\newcommand{\beq}{\begin{equation}}
\newcommand{\bEq}{\end{equation}}

\theoremstyle{plain} 
\newtheorem{theorem}{Theorem}[section]
\newtheorem*{theorem*}{Theorem}
\newtheorem{lemma}[theorem]{Lemma}
\newtheorem*{lemma*}{Lemma}
\newtheorem{corollary}[theorem]{Corollary}
\newtheorem*{corollary*}{Corollary}
\newtheorem{proposition}[theorem]{Proposition}
\newtheorem*{proposition*}{Proposition}
\newtheorem{assumption}[theorem]{Assumption}
\newtheorem*{assumption*}{Assumption}

\newtheorem{definition}[theorem]{Definition}
\newtheorem*{definition*}{Definition}
\newtheorem{example}[theorem]{Example}
\newtheorem*{example*}{Example}
\newtheorem{remark}[theorem]{Remark}

\newtheorem*{remark*}{Remark}
\newtheorem*{remarks*}{Remarks}

\def\author#1{\par
    {\centering{\authorfont#1}\par\vspace*{0.05in}}
}

\def\titlefont{\fontsize{13}{15}\bfseries\boldmath\selectfont\centering{}}
\def\authorfont{\fontsize{13}{15}}

\let\affiliationfont\rhfont

\def\address#1{\par
    {\centering{\affiliationfont#1\par}}\par\vspace*{11pt}
}

\def\body{
\setcounter{footnote}{0}
\def\thefootnote{\alph{footnote}}
\def\@makefnmark{{$^{\rm \@thefnmark}$}}
}

\def\title#1{
    \thispagestyle{plain}
    \vspace*{-14pt}
    \vskip 79pt
    {\centering{\titlefont #1\par}}%
    \vskip 1em
}

\newcommand{\mfct}{m_{\mathrm{fc},t}}
\newcommand{\rfct}{\rho_{\mathrm{fc},t}}
\newcommand{\boeta}{\bm{\eta}}
\newcommand{\bla}{\bm{\lambda}}
\newcommand{\boxi}{\bm{\xi}}
\newcommand{\msc}{m_{\rm sc}}
\newcommand{\f}{f}

\begin{document}
\title{Eigenvector Statistics of Sparse Random Matrices}

\vspace{1.2cm}

\noindent\begin{minipage}[c]{0.33\textwidth}
 \author{Paul Bourgade}
\address{New York University\\
   E-mail: bourgade@cims.nyu.edu}
 \end{minipage}
 \begin{minipage}[c]{0.33\textwidth}
 \author{Jiaoyang Huang}
\address{Harvard University\\
   E-mail: jiaoyang@math.harvard.edu}
 \end{minipage}
\begin{minipage}[c]{0.33\textwidth}
 \author{Horng-Tzer Yau}
\address{Harvard University \\
   E-mail: htyau@math.harvard.edu}

 \end{minipage}

~\vspace{0.3cm}

\begin{abstract}

We prove that the bulk eigenvectors of sparse random matrices, i.e. the adjacency matrices of Erd\H{o}s-R\'enyi  graphs or random regular graphs, are asymptotically jointly normal, provided the averaged degree increases with the size of the graphs. Our methodology follows \cite{BoYaQUE} by analyzing the eigenvector flow under Dyson Brownian motion, combined with an isotropic local law for Green's function. As an auxiliary result, we prove that for the eigenvector flow of Dyson Brownian motion with general initial data, the eigenvectors are asymptotically jointly normal in the direction $\bmq$ after time $\eta_*\ll t\ll r$, if in a window of size $r$, the initial density of states is bounded below and above down to the scale $\eta_*$, and the initial eigenvectors are delocalized in the direction $\bmq$  down to the scale $\eta_*$. 
\end{abstract}



\let\thefootnote\relax\footnote{\noindent The work of P. B. is partially supported by the NSF grant DMS-1513587. The work of H.-T. Y. is partially supported by the NSF grant DMS-1307444, DMS-1606305 and a Simons Investigator award.}

\date{\today}

\vspace{-0.7cm}

\section{Introduction}\label{s:intro}
In this paper, we consider the following two models of sparse random matrices $H$ with sparsity $p=p(N)$:
\begin{enumerate}
\item{(Erd\H{o}s-R\'enyi  Graph Model $G(N, p/N)$)} $H\deq A/\sqrt{p(1-p/N)}$, where $A$ is the adjacency matrix of the Er\H{o}s-R\'enyi graph on $N$ vertices obtained by drawing an edge between each pair of vertices randomly and independently, with probability $p/N$. 

\item{($p$-Regular Graph Model $G_{N,p}$)} $H\deq A/\sqrt{p-1}$, where $A$ is the adjacency matrix of the uniform random $p$-regular graph on $N$ vertices, i.e. a uniformly chosen symmetric matrix with entries in $\{0,1\}$ such that all rows and columns have sum equal to $p$ and all diagonal entries vanish.
\end{enumerate}

Given a graph $\cG$ on $N$ vertices with adjacency matrix $A$, many interesting properties of graphs are revealed by the eigenvalues and eigenvectors of $A$. Such phenomena and the applications have been intensively investigated for over half a century. To mention some, we refer the readers to the books \cite{MR1421568, MR2882891} for a general discussion on spectral graph theory,  the survey article \cite{MR2247919} for the connection between eigenvalues and expansion properties of graphs, and the articles \cite{MR1468791, MR1240959, 609407,1Coifman24052005,2Coifman24052005,doi:10.1137/0611030,MR2952760, MR2294342,MR1450607} on the applications of eigenvalues and eigenvectors in various algorithms, i.e., combinatorial optimization, spectral partitioning and clustering.

We study the spectral properties of sparse random graphs from the random matrix theory point of view, i.e. the local eigenvalue statistics and the eigenvector statistics. It is expected that: \rn{1}) the gap distribution for the bulk eigenvalues $N(\lambda_{i+1}-\lambda_{i})$ is universal, with density approximately given by the \emph{Wigner surmise}; \rn{2}) the distribution of the second largest eigenvalue is given by the Tracy-Widom distribution (the largest eigenvalue of GOE); \rn{3}) the eigenvectors are asymptotically normal. For Wigner type random matrices, it is proved in a series of papers \cite{MR3541852,MR2662426, MR2661171,MR2639734, MR2810797,MR2919197, MR2981427, MR3372074, MR2847916, MR2784665, ly, fix2} for the bulk and \cite{MR1727234, MR2669449, MR2871147} for the edge, that the eigenvalue statistics are universal; it is proved in \cite{MR3034787,MR2930379, BoYaQUE} that the eigenvectors are asymptotically normal. Sparser models are harder to analyze. The bulk universality for both Erd\H{o}s-R\'{e}nyi graphs and random regular graphs in the regime $p\gg 1$ were proved in \cite{MR3098073,MR2964770, MR3429490}. The edge universality was only proved for Erd\H{o}s-R\'{e}nyi graphs in the regime $p\gg N^{1/3}$ in  \cite{MR3098073, MR2964770, Edgesparse}.   Less was known for the distribution of eigenvectors. To our knowledge, only recently, in \cite{normalregular}, Backhausz and Szegedy proved that the components of almost eigenvectors of $p$-regular graphs with fixed $p$ converges to normal distribution in weak topology. However the proof heavily depends on the special structure of regular graphs and is hard to be generalized to other models.

Let $H$ be the normalized adjacency matrix of $G(N,p/N)$ or $G_{N,p}$ in the sparse regime, i.e. $p=p(N)\ll N$. We denote its eigenvalues as $\lambda_1\leq \lambda_2\leq \cdots \leq \lambda_N$ and the corresponding normalized eigenvectors $\bmu_1,\bmu_2,\cdots, \bmu_N$. The main goal of this paper is to prove that the bulk eigenvectors for $H$ in the regime $p\gg 1$ are asymptotically jointly normal. Comparing with \cite{normalregular}, our results give explicitly the variance of the limit distribution, the asymptotical normality holds in any direction, and the argument does not depend on the special symmetry of the models.

\begin{theorem}\label{t:evsparse}
Fix arbitrary small constant $\delta,\kappa>0$. Let $H$ be the normalized adjacency matrix of sparse Erd\H{o}s-R\'enyi graphs $G(N, p/N)$ with sparsity $N^{\delta}\leq p\leq N/2$; or 
the normalized adjacency matrix of $p$-regular graphs $G_{N,p}$ with sparsity $N^{\delta}\leq p\leq N^{2/3-\delta}$.
Fix a positive integer $n>0$ and a polynomial $P$ of $n$ variables.
Then for any unit vector $\bmq\in \bR^N$, such that $\bmq \perp \bme$ (where $\bme=(1,1,\cdots, 1)^*/\sqrt{N}$), and deterministic indexes $i_1,i_2,\cdots, i_n\in \qq{\kappa N, (1-\kappa)N}$, 
there exists a constant $\fd>0$ depending on $\delta$ such that
\begin{align}\label{evuni}
\begin{split}
\left|\bE[ P(N \langle {\bmq}, {\bmu}_{i_1}\rangle^2, N \langle {\bmq}, {\bmu}_{i_2}\rangle^2, \cdots, N \langle {\bmq}, {\bmu}_{i_n}\rangle^2)]
-\bE[ P({\mathscr{N}}_1^2,{\mathscr{N}}_2^2, \cdots, {\mathscr{N}}_n^2)]\right|\leq CN^{-\fd},
\end{split}
\end{align}
provided $N$ is large enough, where $\bmu_i$ are eigenvectors of $H$, $\mathscr{N}_i$ are independent standard normal random variables.
\end{theorem}

In particular, Theorem \ref{t:evsparse} implies that the entries of eigenvectors are asymptotically independent Gaussian. Indeed, for any fixed $\ell\in\mathbb{N}$ and deterministic $i\in\qq{\kappa N, (1-\kappa)N}$, $\alpha_1,\dots, \alpha_\ell\in\qq{1,N}$, possibly depending on $N$, 
we have $\sqrt{N}(u_i(\alpha_1),\dots,u_i(\alpha_\ell))\to(\mathscr{N}_1,\dots,\mathscr{N}_\ell)$, a vector with independent normal entries (provided the sign of the first entry of $u_i$, say, is uniformly and independently chosen).

The proof of Theorem \ref{t:evsparse} consists of three steps, analogous to the three-step strategy developed in a series of papers \cite{MR2810797,MR2919197, MR2981427,MR3429490} for proving bulk eigenvalue universality:
\begin{enumerate}
\item Establish the (isotropic) local semicircle law for sparse random matrices down to the optimal scale $(\log N)^C/N$.
\item Analyze the eigenvector flow of Dyson Brownian motion to derive asymptotical normality of eigenvectors for sparse random matrices with a small Gaussian component. 
\item Prove by comparison that the eigenvector statistics of sparse random matrices are the same as those of ones with a small Gaussian component.  
\end{enumerate}

For the first step, the local semicircle laws for sparse random matrices were established in \cite{MR3098073} for Erd\H{o}s-R\'{e}nyi graphs, and in \cite{Rigidregular} for $p$-regular graphs. For the third step, a robust comparison argument was developed in \cite{MR3429490}, and our case follows directly. The main content of this paper is the second step.  We study the eigenvector flow of Dyson Brownian motion with general initial data. For any $N\times N$ real deterministic matrix $H$, we define the following random matrix process, the Dyson Brownian motion
\begin{align}
\label{DBM} \rd h_{ij}(t) =\rd w_{ij}(t)/\sqrt{N},
\end{align}
where $W_t=(w_{ij}(t))_{1\leq i,j\leq N}$ is symmetric with $(w_{ij}(t))_{1\leq i\leq j\leq N}$ a family of independent Brownian motions of variance $(1+\delta_{ij})t$. We denote $H_t= (h_{ij}(t))_{1\leq i,j\leq N}$, and  $H_0=H$ is our original matrix. We denote the eigenvalues of $H_t$ as $\bm\la(t): \lambda_1(t)\leq \lambda_2(t)\leq \cdots \leq \lambda_N(t)$ and the corresponding eigenvectors $\bmu_1(t),\bmu_2(t),\cdots, \bmu_N(t)$, where we write the $j$-th entry of $\bmu_i(t)$ as $u_{ij}(t)$. 

Under some mild local regularity conditions (see Assumption \ref{a:boundImm} and \ref{a:boundEv}) on the initial matrix $H_0$, we first prove the isotropic local law for the Green's function of $H_t$, which is a consequence of the small Gaussian component. The isotropic local law was first proved for Wigner matrices and sample covariance matrices in \cite{MR3103909, MR3183577}. Our result provides a dynamical version of the entry-wise local law, with general initial matrices. With the isotropic local law as input combined with the rigidity estimates for eigenvalues from \cite{ly}, we analyze the \emph{eigenvector moment flow}, introduced in \cite{BoYaQUE}. We prove that the eigenvectors of $H_t$ corresponding to ``bulk" eigenvalues are asymptotically normal after a short time. Our result can be viewed as an extension of \cite[Theorem 7.1]{BoYaQUE}. We require only weak local information of the initial data, while \cite{BoYaQUE} relies on a stronger form of the local law, including  rigidity of the eigenvalues.

In this paper motivated by random graphs, we restrict our attention to symmetric matrices, but our method equally applies to the Hermitian universality class.

\subsection{Preliminary notations} \label{s:prenot}A fundamental quantity in this paper is  the resolvent of $H_t$, denoted 
$ G(t;z):=(H_t-z)^{-1}$,
and the Stieltjes transform
\begin{align*} 
m_t(z):=\frac{1}{N}\Tr G(t,z)=\frac{1}{N}\sum_{i=1}^{N}\frac{1}{\lambda_i(t)-z},
\end{align*}
where $z\in \bC_+$ in in the upper half complex plane. Often, we will write $z$ as the sum of its real and imaginary parts
$z = E + \i \eta$ where $E = \Re[z],\eta = \Im [z]$.

We denote by $\rfct$ the free convolution of the empirical eigenvalue distribution of $H_0$, i.e. $\rho_0=1/N\sum \delta_{\lambda_i(0)}$ and the semicircle law with variance $t$, and  $\mfct$ the Stieltjes transform of $\rfct$. The density $\rfct$ is analytic on its support for any $t>0$. The function $\mfct$ solves the equation
 \begin{equation}\label{def_gi}
\mfct (z) = m_0 ( z  + t \mfct (z) ) = \frac{1}{N}\sum_{i=1}^N g_i(t,z),\quad g_i(t,z)\deq \frac{1}{ \lambda_i (0) - z - t \mfct (z) },
\end{equation}  
where refer to \cite{MR1488333} for a detailed study of free convolution with semi-circle law.
For any $t\geq 0$, we denote the classical eigenvalues of $\rfct$ by $\gamma_i(t)$, which is given by 
\beq\label{eqn:eigloc}
\gamma_i(t)=\sup_{x}\left\{\int_{-\infty}^x \rfct(x)\rd x \geq \frac{i}{N}\right\},\quad i\in\qq{1,N}. 
\bEq

Throughout the paper we use the following notion of { overwhelming probability}.
\begin{definition}  We say that a family of events ${\cal F} (u)$ indexed by some parameter(s) $u$ holds with overwhelming probability, if for any large $D>0$ and $N\geq N(D,u)$ large enough,
\begin{align}
\bP[ {\cal F} (u) ] \geq 1 - N^{-D},
\end{align}
uniformly in $u$.

\end{definition}

We use $C$ to represent a large universal constant, and $c$ for a small universal constant, which may depend on other universal constants, i.e., $\fc$ in the control parameter $\psi$ defined in \eqref{control}, constants $\frak a$  and $\frak b$ in Assumption \ref{a:boundImm} and \ref{a:boundEv}, and may be different from line by line. We write $X\lesssim Y$ or $X=O(Y)$, if there exists some universal constant $C$ \nc such that $|X|\leq CY$. We write $X\lesssim_k Y$, or $X=O_k(Y)$ if there exists some constant $C_k$, which only depends on $k$ (and possibly other universal constants), such that $|X|\leq C_kY$. We write $X\ll Y$ if there exists some small constant $c$, such that $N^{c}|X|\leq Y$.

We now can state the assumptions on the initial matrix $H_0$. In Sections \ref{s:locallaw} and \ref{relaxation}, we fix an arbitrarily small number $\frak c>0$, and define the control parameter 
\begin{align}\label{control}
\psi= N^{\frak c}.
\end{align}
We fix an energy level $E_0$, radius $1/N\ll r\leq 1$, and mesoscopic scales $1/N\ll \eta_*\ll r$, where $r$ and $\eta_*$ will depend on $N$. For example, the reader can take $\eta_*=\psi/N$, $r=N^{-1/2}$ in mind. We will study the eigenvectors corresponding to the ``bulk'' eigenvalues, which refer to eigenvalues on the interval $[E_0-r,E_0+r]$. We show that after short time, the projections of those ``bulk" eigenvectors on some unit vector $\bmq$ are asymptotically normal.

 The first assumption is the same as in \cite{ly}, which imposes the regularity of density of $H_0$ around $E_0$.

\begin{assumption}{\label{a:boundImm}}
 We assume that there exists some large constant $\frak a>0$ such that 
 \begin{enumerate}
  \rm\item \label{boundnorm} The norm of $H_0$ is bounded, $\|H_0\|\leq N^{\frak{a}}$.
  \rm\item \label{locallaw} The Stieltjes transform of $H_0$ is lower and upper bounded
  \begin{align}\label{e:imasup}
   \frak{a}^{-1}\leq \Im[m_0(z)]\leq \frak{a},
  \end{align}
  uniformly for any $z\in \{E+\i \eta: E\in[E_0-r, E_0+r], \eta_*\leq \eta\leq 1\}$.
 \end{enumerate}
\end{assumption}

Besides the information on eigenvalues of the initial matrix $H_0$, we also need the following regularity assumption on its eigenvectors. 

\begin{assumption}{\label{a:boundEv}}
We assume that for some unit vector  $\bmq$, there exists some small constant $\frak b>0$ such that
  \begin{align}
  \label{e:delocalize}\big |  \langle \bmq, G(0,z)\bmq\rangle - m_0(z) \big | \leq    N^{-\frak b},
 \end{align}
 uniformly for any $z\in \{E+\i \eta: E\in[E_0-r, E_0+r], \eta_*\leq \eta\leq r\}$, where $m_0$ is the Stieltjes transform of $H_0$.

\end{assumption}

\subsection{Statement of Results}
Let $E_0$ and $r$ be the same as in Assumption \ref{a:boundImm}. For any $0\leq \kappa< 1$, 
we denote
\begin{align*}
 I^r_\kappa (E_0):=[E_0-(1-\kappa)r, E_0+ (1-\kappa)r], 
\end{align*}
and the spectral domain: 
\begin{align}\label{e:dom}
 \cal D_\kappa:=\{z=E+\i \eta: E\in I^r_\kappa (E_0), \psi^4/N\leq \eta \leq 1-\kappa r \}.
\end{align}

\begin{theorem}\label{t:normal}
 We assume that the inital matrix $H_0$ satisfies Assumption \ref{a:boundImm} and \ref{a:boundEv}. 
Fix $\kappa>0$, a positive integer $n>0$ and a polynomial $P$ of $n$ variables. Then for any $\eta_*\ll t\ll r$, there exists a constant $\fd>0$ depending on $\fa,\fb, r,t$ such that 
\begin{equation}\label{eqn:detRelax}
\sup_{ |I|=n:\atop\forall k\in I, \lambda_k(t)\in I_{2\kappa}^r(E_0)}
\left|\bE\left[P\left(\left(N|\langle \bmq,\bmu_{k} (t)\rangle|^2\right)_{k\in I}\right) \right]-\bE\left[ P\left((|\mathscr{N}_j|^2)_{j=1}^n\right)\right]\right|\leq CN^{-\frak{d}},
\end{equation}
provided $N$ is large enough, where $\sup$ is over all possible index sets $I$, and $\mathscr N_j$ are independent standard normal random variables.
\end{theorem}

As a corollary, we have the following local quantum unique ergodicity statements for ``bulk" eigenvectors.
\begin{corollary}\label{c:que}
 We assume that the initial matrix $H_0$ satisfies Assumption \ref{a:boundImm}. We further assume that there exists a
small  constant $ \frak b$ such that
 \begin{align}
 \label{e:delocalize}\big |  (H_0-z)^{-1}_{ij} - m_0(z) \delta_{ij}  \big | \leq    \frac{1}{N^{\frak b}},    \quad m_0(z) = \frac 1 N \Tr (H_0-z)^{-1} 
\end{align}
uniformly for any $z\in \{E+\i \eta: E\in[E_0-r, E_0+r], \eta_*\leq \eta\leq r \}$. Then the following quantum unique ergodicity holds: Fix $ \kappa >0$. For any $\eta_*\ll t\ll r$ and $\e>0$, there exists a constant $\fd>0$ depending on $\fa,\fb, r,t$ such that 
\beq\label{localque}
\sup_{ k: \lambda_k(t)\in I_{2\kappa}^r(E_0) }  \bP\left(\left|\frac{N}{\|{\bf a}\|_1}\sum_{i=1}^N a_i u_{ki}^2\right|>N^{-\e} \right)\leq C\, N^{2\e}\left(N^{-\fd}+\|{\bf a}\|_1^{-1}\right),
\bEq
provided $N$ is large enough, where ${\bf a}=(a_1,a_2,\cdots, a_N)$, such that $\sum_i a_i=0$ and $\max_i |a_i| \le 1$,  and its norm $\|{\bf a}\|_1=\sum |a_i|$.
\end{corollary}

\noindent\textbf{Acknowledgements.}
The authors thank Antti Knowles for pointing out an error in an early version of this paper.

\section{Local Law} \label{s:locallaw}

In this section, we prove the following isotropic local law for the resolvent of $H_t$. We write $H_0=U_0 \Lambda_0 U_0^*$, where $\Lambda_0=\diag\{\lambda_1(0), \cdots, \lambda_N(0)\}$, and $U_0$ is the orthogonal matrix of its eigenvectors. Theorem \ref{t:isolaw} states that $G(t,z)$ is well approximated by $U_0\diag\{g_1(t,z), g_2(t,z),\cdots, g_N(t,z)\}U_0^*$ where $g_i$ are defined in \eqref{def_gi}. It implies that the Green function becomes regular after adding a small Gaussian component.
\begin{theorem}{\label{t:isolaw}}
 Under the Assumption \ref{a:boundImm}, fix $\kappa>0$. Then for any $\eta_*\ll t\ll r$ and any unit vector $\bmq=(q_1,q_2,\cdots, q_N)^*\in \bR^N$,  
uniformly for any $z\in \cal D_\kappa$ (as in \eqref{e:dom}),  the following holds with overwhelming probability, 
 \begin{align}\label{e:isolaw}
 \left|\langle \bmq, G(t,z)\bmq\rangle - \sum_{i=1}^{N}\langle \bmu_i(0), \bmq\rangle^2g_i(t,z)\right|\leq \frac{\psi^2}{\sqrt{N\eta}}\Im\left[\sum_{i=1}^{N}\langle \bmu_i(0), \bmq\rangle^2g_i(t,z)\right] ,
 \end{align}
 provided $N$ is large enough, where $\bmu_i(0)$ are eigenvectors of $H_0$, and $g_i$ are defined in \eqref{def_gi}.
\end{theorem}

\subsection{Rigidity of Eigenvalues}
In \cite{ly}, the eigenvalues of $H_t$ are detailed studied under the Assumption \ref{a:boundImm}. In this section we recall some estimates on the locations of eigenvalues from \cite{ly}. For the free convoluted density $\rho_{{\rm fc},t}$, we have the following deterministic estimate on its Stieltjes transform and classical eigenvalue locations (as in \eqref{def_gi} and \eqref{eqn:eigloc}) from \cite[Lemma 7.2]{ly}.
\begin{proposition}{\label{p:mfc}}
Under the Assumption \ref{a:boundImm}, fix $\kappa>0$. Then for any $\eta_*\ll  t\ll r$ and  $N$ 
large enough, the following holds: uniformly for $z\in \{E+\ri \eta: E\in I_\kappa^r(E_0), 0<\eta\leq 1-\kappa r\}$, the Stieltjes transform $\mfct$, 
 \begin{align}\label{e:dbound}
  C^{-1}\leq \Im[\mfct(z)]\leq C,
 \end{align}
 and
 \begin{align}\label{e:sumgibound}
  |\mfct(z)|\leq \frac{1}{N}\sum_{i=1}^N |g_i(t,z)|\leq C\log N,
 \end{align}
where $C$ is a constant depending on the constant $\frak{a}$ in Assumption \ref{a:boundImm}, and $g_i(t,z)$ are as in \eqref{def_gi}; for the classical eigenvalue locations, uniformly for any index $i$ such that $\gamma_i(t)\in I_r^{\kappa}(E_0)$, we have
\begin{align}\label{e:dergamma}
|\del_t \gamma_{i}(t)| \leq C\log N. 
\end{align}
 \end{proposition}
\begin{proof}
\eqref{e:dbound} is the same as \cite[(7.7) Lemma 7.2]{ly}.  For \eqref{e:sumgibound}, we denote $\tilde E+\i\tilde\eta := z+t\mfct(z)$, and divide the sum into the following dyadic regions: 
\begin{align*}
U_{0}=\{i: |\lambda_i(0)-\tilde E|\leq \tilde \eta\},\quad U_n=\{i: 2^{n-1}\tilde \eta< |\lambda_i(0)-\tilde E|\leq 2^n \tilde \eta \}, \quad 1\leq n\leq \lceil-\log_2(\tilde \eta)\rceil.
\end{align*}
For the eigenvalues which do not belong to $\cup_n U_n$, we have $|\lambda_i(0)-\tilde{E}|\geq1$.  Since $\tilde \eta\gtrsim t\gg \eta_*$, we have
\begin{align*}
|U_n|\leq \sum_{i=1}^N\frac{2(2^n\tilde\eta)^2}{|\lambda_i(0)-\tilde E- \ri2^n \tilde \eta |^2}\leq 2\Im[m_0(\tilde{E}+2^n\ri\tilde{\eta})]2^{  n}\tilde\eta N\leq C2^{  n} \tilde\eta N.
\end{align*}
Thus we can bound  \eqref{e:sumgibound} 
\begin{align}\label{e:dyadic}
\frac{1}{N}\sum_{i=1}^{N}|g_i|
\leq &  \frac{1}{N}\sum_{n=0}^{\lceil \log_2 N\rceil}\sum_{i\in U_n}\frac{1}{|\lambda_i(0)-\tilde E-\ri\tilde \eta|} +1
\leq \frac{1}{N}\sum_{n=0}^{\lceil- \log_2 \tilde \eta\rceil}\frac{|U_n|}{2^{n-1}\tilde \eta}+1\leq C\log N.
\end{align}
Finally for \eqref{e:dergamma}, we have $|\del_t \gamma_i(t)|=|\Re[\mfct(\gamma_i(t))]|\leq C\log N$.
\end{proof}

The following result on eigenvalue rigidity estimates of $H_t$ is from \cite[Theorem 3.3]{ly}. 
\begin{theorem}{\label{t:rig}}
Under the Assumption \ref{a:boundImm}, fix $\kappa>0$. Then for any  $\eta_*\ll t\ll r$, and  $N$ 
large enough, with overwhelming probability, the followings hold:
 \begin{align}\label{eqn:locallaw}
  |m_t(z)-\mfct(z)|\leq \psi (N\eta)^{-1}
 \end{align}
 uniformly for $z\in \cal D_\kappa$; and for the eigenvalues, 
 \begin{align*}
 |\lambda_{i}(t)-\gamma_{i}(t)|\leq\psi N^{-1},
 \end{align*}
 uniformly for any index $i$ such that $\lambda_i(t)\in I_r^{\kappa}(E_0)$.
\end{theorem}

\subsection{Isotropic Local Law}

Before we start proving Theorem \ref{t:isolaw}, we need some reductions. We write $H_0$ as $H_0=U_0 \Lambda_0 U_0^*$, where $\Lambda_0=\diag\{\lambda_1(0), \cdots, \lambda_N(0)\}$, and $U_0$ is the orthogonal matrix of its eigenvectors. Since $H_t\stackrel{d}{=}H_0+\sqrt{t}W$, where $W$ is a standard Gaussian orthogonal ensemble, i.e., $W=(w_{ij})_{1\leq i\leq j\leq N}$ is symmetric with $(w_{ij})_{1\leq i\leq j\leq N}$ a family of independent Brownian motions of variance $(1+\delta_{ij})/N$, we have the following equality in law:
\begin{align*}
\langle \bmq, G(t,z)\bmq \rangle
=& \langle \bmq,(U_0\Lambda_0U_0^*+\sqrt{t}W-z)^{-1}\bmq \rangle
=\langle \bmq,U_0(\Lambda_0+\sqrt{t}U_0^*WU_0-z)^{-1}U_0^*\bmq \rangle \\
\stackrel{d}{=}&
\langle \bmq, U_0(\Lambda_0+\sqrt{t}W-z)^{-1}U_0^* \bmq\rangle=\langle U_0^*\bmq, (\Lambda_0+\sqrt{t}W-z)^{-1} U_0^*\bmq\rangle.
\end{align*}
Therefore, Theorem \ref{t:isolaw} can be reduced to the case that $H_t=\Lambda_0+\sqrt{t}W$:
  \begin{align}
 \left|\langle \bmq, G(t,z)\bmq\rangle - \sum_{i=1}^{N}q_i^2g_i(t,z)\right|\leq \frac{\psi^2}{\sqrt{N\eta}}\Im\left[\sum_{i=1}^{N}q_i^2g_i(t,z)\right] .
 \end{align}

The entry-wise local law of the matrix ensemble $\Lambda_0+\sqrt{t}W$ (so called deformed Gaussian orthogonal ensemble) was studied in \cite{ly}. In the following we recall some estimates on the entry-wise local law from \cite[Theorem 3.3]{ly}. To state it we need to introduce some notations. For any index set $\bT\in\qq{1,N}$, we denote $[H_t]_{i,j\notin\bT}$ the minor of $H_t$ by removing the columns and rows indexed by $\bT$, and its resolvent by $G^{(\bT)}(t,z):=([H_t]_{i,j\notin \bT}-z)^{-1}$. Recall the definition of $g_i$ from \eqref{def_gi}:
\begin{align}\label{def:gi}
g_i(t,z)=\frac{1}{\lambda_i(0)-z-t\mfct(z)}.
\end{align}
For the simplicity of notation, if the context is clear, we may simply write $g_i(t,z)$ as $g_i$. Roughly speaking, the following theorem states that the resolvent matrix $G(t,z)$ is close to the diagonal matrix $\diag\{g_1,g_2,\cdots, g_{N}\}$.
 
\begin{theorem}\label{p:resoentry}
The initial matrix $H_0=\mathrm{diag}\{\lambda_1(0),\lambda_2(0),\cdots, \lambda_N(0)\}$ satisfies Assumption \ref{a:boundImm} and fix $\kappa>0$. Then for any $\eta_*\ll t\ll r$ and $N$ 
large enough, with overwhelming probability, the following hold. Uniformly for any $z\in \cal D_\kappa$: for the diagonal resolvent entries,
\begin{align}\label{e:diagbound}
\left|G_{ii}^{(\bT)}(t,z)-g_i(t,z)\right| \leq \frac{\psi t}{\sqrt{N\eta}}|g_i(t,z)|^2,
\end{align}
and for the off-diagonal resolvent entries,
\begin{align}\label{e:Gij}
\left|G_{ij}^{(\bT)}(t,z)\right|\leq \frac{\psi}{\sqrt{N\eta}}\min\{|g_i(t,z)|,|g_j(t,z)|\}\leq \frac{\psi}{\sqrt{N\eta}}\left(|g_i(t,z)||g_j(t,z)|
\right)^{1/2},
\end{align}
where $\bT$ is any index set of size $|\bT|\leq \log N$.
\end{theorem}

\begin{proof}[Proof of Theorem \ref{t:isolaw}] From the discussions above, we can assume that $H_0=\Lambda_0$ is diagonal, and take $H_t=\Lambda_0+\sqrt{t}W$, where $W$ is the standard Gaussian orthogonal ensemble.
The quadratic term in \eqref{e:isolaw} can be written as a sum of diagonal terms and off-diagonal terms:
\begin{align*}
\langle \bmq, G(t,z)\bmq\rangle = \sum_{i=1}^{N}G_{ii}q_i^2+\sum_{i\neq j} G_{ij}q_iq_j,
\end{align*}
where $\bmq=(q_1,q_2,\cdots,q_N)$. The proof consists of two parts, the first part is trivial, we prove that the leading order term is the sum over diagonal terms; the second part is more involved, we show that the sum over off-diagonal terms is negligible by moment method.

For the diagonal terms, from \eqref{e:diagbound} in Theorem \ref{p:resoentry} and \eqref{e:highmoment} in Proposition \ref{p:sumqg}, with overwhelming probability we have
\begin{align}\label{e:diagsum}
\left|\sum_{i=1}^{N}G_{ii} q_i^2 -\sum_{i=1}^{N}g_iq_i^2\right|\leq \frac{\psi t}{\sqrt{N\eta}}\sum_{i=1}^{N}|g_i|^2q_i^2
\leq \frac{\psi}{\sqrt{N\eta}}\Im\left[\sum_{i=1}^N q_i^2 g_i\right]
\end{align}

For the second part we prove that for any integer $k>0$, uniformly for $z\in \cal D_\kappa$, we have
\begin{align}\label{e:moment}
\bE\left[|\cal Z|^{2k}\right] \lesssim_k {\cal Y}^{2k}, \quad \cal Z=\sum_{i\neq j}G_{ij}q_i q_i, \quad \cal Y=\frac{\psi\log N}{\sqrt{N\eta}}\Im\left[\sum_{i=1}^N q_i^2g_i\right] .
\end{align}
where the implicit constant depends only on $k$. Then it will follows from the Markov inequality that $|\cal Z|\leq \psi^{2}\Im[\sum_{i=1}^N q_i^2g_i]/\sqrt{N\eta}$ holds with overwhelming probability. By Assumption \ref{a:boundImm}, we have the following trivial lower bound for $\Im\left[\sum_iq_i^2 g_i\right]$,
\begin{align}\label{e:trivialestIm}
\Im\left[\sum_{i=1}^Nq_i^2 g_i\right]
=\sum_{i=1}^N\frac{(\eta+t\Im[\mfct(z)])q_i^2}{|\lambda_i(0)-z-t\mfct(z)|^2}\gtrsim \frac{\eta}{N^{2\fa}}.
\end{align}

We expand $\bE[|\cZ|^{2k}]$, and introduce the shorthand notation $X_{b_{2i-1}b_{2i}}:=G_{b_{2i-1}b_{2i}}$ for $1\leq i \leq k$, and $X_{b_{2i-1}b_{2i}}:=G^*_{b_{2i-1}b_{2i}}$ for $k+1\leq i \leq 2k$,
\begin{align}\label{expand1}
\bE\left[|\cal Z|^{2k}\right]
=\sum_{\bmb} q_{b_1}q_{b_2}\cdots q_{b_{4k}}\bE[X_{b_1b_2}X_{b_3b_4}\cdots X_{b_{4k-1}b_{4k}}],
\end{align}
where $\bmb=(b_1,b_2,\cdots, b_{4k})$ and the sum $\sum_{\bmb}$ is over all $\bmb$'s such that  $b_{2i-1}\neq b_{2i}$, for $1\leq i\leq 2k$. To obtain an efficient control on $\bE[X_{b_1b_2}X_{b_3b_4}\cdots X_{b_{4k-1}b_{4k}}]$, we need to understand the correlations between these off-diagonal resolvent entries $G_{ij}$ for $i\neq j$. Heuristically, $G_{ij}$ mainly depends on the matrix entry $h_{ij}$, weakly depends on the matrix entries on the same row and column, and the dependence on the rest of the matrix $H$ is negligible. Therefore the correlations of $G_{ij}$ and $G_{mn}$ are negligible if $\{i,j\}\cap\{m,n\}=\emptyset$. In the rest of this section, we will make this heuristic argument more rigorous.

We denote the index set $\bT=\{b_1,b_2,\cdots,b_{4k-1},b_{4k}\}$. Recall the following Schur complement formula
\begin{align*}
\left(
\begin{array}{cc}
A-z & B^*\\
B & C-z
\end{array}
\right)^{-1}
=\left(
\begin{array}{cc}
\left(A-z-B^*(C-z)^{-1}B\right)^{-1} & *\\
*&*
\end{array}
\right)
\end{align*}
where $A$, $B$ and $C$ are block matrices. We take $A=[H_t]_{i,j \in \bT}$, $B=[H_t]_{i\notin \bT, j\in \bT}$ and $C=[H_t]_{i\notin \bT, j\notin \bT}$, where $[H_t]_{i,j\in \bT}$ is the submatrix of $H_t$ with row and column indices $i,j\in \bT$, and $[H_t]_{i\notin\bT,j\in \bT}$ and $[H_t]_{i,j\notin \bT}$ are defined analogously. Recall that $G^{(\bT)}(t,z)$ is the resolvent of the submatrix $[H_t]_{i,j\notin \bT}$ and $m^{(\bT)}_t(z)=\Tr G^{(\bT)}/N$ is its Stieltjes transform. Schur complement formula gives the following resolvent identity: 
\begin{align*}
[G]_{i,j\in \bT}
=&\left([H_t]_{i,j\in \bT}-z-[H_t]_{i\notin \bT,j\in \bT}^*G^{(\bT)}[H_t]_{i\notin \bT, j\in \bT} \right)^{-1}\\
=&\left([\Lambda_0]_{i,j\in \bT}+\sqrt{t}[W]_{i,j\in \bT}-z-t\left([W]_{i\notin \bT,j\in \bT}^*G^{(\bT)}[W]_{i\notin \bT, j\in \bT}\right)\right)^{-1}
=:(D(z)-\cal E(z))^{-1},
\end{align*}
where $D(z)$ and $\cal E(z)$ are two $|\bT|\times |\bT|$ matrices, which depend on the index set $\bT$,
\begin{align}\begin{split}\label{defcE}
&D=[\Lambda_0]_{i,j\in \bT}-z-t\mfct, \quad \cal E=\cal E^{(1)}+\cal E^{(2)}+\cal E^{(3)},\quad \cal E^{(1)}=t\left(m_t^{(\bT)}-\mfct\right)\\
&\cal E^{(2)}=-\sqrt{t}[W]_{i,j\in \bT},\quad \cal E^{(3)}=t\left([W]_{i\notin \bT,j\in \bT}^*G^{(\bT)}[W]_{i\notin \bT, j\in \bT} -m_t^{(\bT)}\right).
\end{split}\end{align}
With overwhelming probability, uniformly for any $z\in \cal D_\kappa$, the error term $\cal E(z)$ is much smaller than $D(z)$ in the sense of matrix norm. In fact, for $\cal E^{(1)}$, by \eqref{eqn:locallaw} and notice the deterministic estimate from interlacing of eigenvalues $|m_t-m^{(\bT)}_t|\leq |\bT|/N\eta$, with overwhelming probability, $t|m_t^{(\bT)}-\mfct|\leq  \psi t/(N\eta)$. For $\cal E^{(2)}$, with overwhelming probability, its entries are uniformly bounded by $\psi(t/N)^{1/2}$. For $\cal E^{(3)}$, with overwhelming probability, we have the following estimate
\begin{align*}
\cal E^{(3)}_{mn}=t\sum_{ij}\left(w_{mi}w_{nj}-\frac{\delta_{ij}\delta_{mn}}{N}\right)G_{ij}^{(\bT)}\leq \frac{\psi t}{N}\sqrt{\sum_{ij}|G_{ij}^{(\bT)}|^2}=\psi t\frac{\Im[ m^{(\bT)}_t]^{1/2}}{\sqrt{N\eta}}\lesssim \frac{\psi t}{\sqrt{N\eta}},
\end{align*}
where the first inequality follows from the large deviation estimate \cite[Appendix B]{MR2981427}, and the second inequality follows from \eqref{eqn:locallaw}. Since $\cal E(z)$ is a $|\bT|\times |\bT|$ matrix, where $|\bT|\leq 4k$, and with overwhelming probability, its entries are uniformly bounded, so is its norm:
$
\|\cal E(z)\|\lesssim_k \psi (t+\eta)/\sqrt{N\eta}
$. 
For $z\in \cal D_\kappa$, we have $\Im[z+t\mfct(z)]\gtrsim (\eta+t)$, which implies $\|D(z)\|\gtrsim (\eta+t)$. As a result, there exists a constant $C_k$ which depends only on $k$, the following holds: uniformly for  any $z\in \cal D_\kappa$, 
\begin{align}\label{e:errorcomp}
\|\cal E(z)\|\leq  C_k\frac{\psi}{\sqrt{N\eta}}\|D(z)\|
\end{align}
with overwhelming probability.  We define the event $\cal A$, such that \eqref{e:errorcomp} holds. Since it holds with overwhelming probability, for sufficiently large $N$, we can assume that 
\begin{align}\label{e:Acmeasure}
\bP(\cal A^c)\leq N^{-(4\fa +6)k}. 
\end{align}
By Taylor expansion, on  the event $\cal A$, we have
\begin{align*}
[G]_{i,j\in \bT}=(D-\cal E)^{-1}
=\sum_{\ell=0}^{f-1} D^{-1}\left(\cal E D^{-1}\right)^{\ell}
+(D-\cal E)^{-1}\left(\cal E D^{-1}\right)^f,
\end{align*}
where $f$ is a large number, and we will choose it later. In the rest of the proof, we denote 
\begin{align*}
G^{(\ell)}:=D^{-1}\left(\cal E D^{-1}\right)^{\ell}, \quad 0\leq \ell \leq f-1, \quad G^{(\infty)}:=(D-\cal E)^{-1}\left(\cal E D^{-1}\right)^f,
\end{align*}
For $1\leq \ell \leq f-1$ or $\ell=\infty$, we define $X^{(\ell)}_{b_{2i-1}b_{2i}}:=G^{(\ell)}_{b_{2i-1}b_{2i}}$ for $1\leq i\leq k$, and $X^{(\ell)}_{b_{2i-1}b_{2i}}:=G^{(\ell)*}_{b_{2i-1}b_{2i}}$ for $k+1\leq i\leq 2k$. We remark that $G^{(\ell)}$ and $X^{(\ell)}$ implicitly depend on the index set $\bT$. With these notations, we have
\begin{align*}
X_{b_{2i-1}b_{2i}}=\sum_{\ell=1}^{f-1}X^{(\ell)}_{b_{2i-1}b_{2i}}+X^{(\infty)}_{b_{2i-1}b_{2i}},\quad 1\leq i\leq 2k,
\end{align*}
where we used the fact that $b_{2i-1}\neq b_{2i}$, and $D^{-1}=\diag\{g_i\}_{i\in \bT}$ is a diagonal matrix; therefore, when $\ell=0$, the term $X_{b_{2i-1}b_{2i}}^{(0)}$ vanishes. On the event $\cal A$, $\|(D-\cal E)^{-1}\|\lesssim_k1/\eta$ and $\|\cal E D^{-1}\|\lesssim_k \psi/(N\eta)^{1/2}$, they together imply:
\begin{align*}
\left|X^{(\infty)}_{b_{2i-1}b_{2i}}\right|\lesssim_k \frac{1}{\eta}\left(\frac{\psi}{\sqrt{N\eta}}\right)^f
\end{align*}
In the following we show that: once we take $f$ sufficiently large, these terms $X^{(\infty)}_{b_{2i-1}b_{2i}}$ are negligible, and do not contribute to \eqref{e:moment}.  Since $X_{b_{2i-1}b_{2i}}$'s are all uniformly bounded by $1/\eta$, and the sum $\sum_{i=1}^{N} |q_i|$ is trivially bounded by $N^{1/2}$, we have
\begin{align}
\bE\left[|\cal Z|^{2k}\right]
=&\sum_{\bmb} q_{b_1}q_{b_2}\cdots q_{b_{4k}}\bE[X_{b_1b_2}X_{b_3b_4}\cdots X_{b_{4k-1}b_{4k}}1_{\cal A}]+O\left(N^{2k}\eta^{-2k}\bP(\cal A^{c})\right)
\end{align}
By our choice of set $\cA$, i.e. \eqref{e:Acmeasure}, combining with the estimate \eqref{e:trivialestIm}, we have $N^{2k}\eta^{-2k}\bP(\cal A^{c})\leq (N^{2\fa+2}\eta)^{-2k}\leq \cY^{2k}$. Therefore, 
\begin{align}\label{eqn:comperror}
\bE\left[|\cal Z|^{2k}\right]=&\sum_{\bmb} q_{b_1}q_{b_2}\cdots q_{b_{4k}}\bE[X_{b_1b_2}X_{b_3b_4}\cdots X_{b_{4k-1}b_{4k}}1_{\cal A}]+O\left(\cY^{2k}\right),
\end{align}
where $\cY$ is as in \eqref{e:moment}.
We separate the leading term of the product of those $X_{b_{2i-1}b_{2i}}$ as
\begin{align}\label{e:decompsum}
X_{b_1b_2}X_{b_3b_4}\cdots X_{b_{4k-1}b_{4k}}=\prod_{i=1}^{2k}\sum_{\ell=1}^{f-1}X^{(\ell)}_{b_{2i-1}b_{2i}}+\sum_{j=1}^{2k}\left(\prod_{i=1}^{j-1}X_{b_{2i-1}b_{2i}} \right)X^{(\infty)}_{b_{2j-1}b_{2j}}\left( \prod_{i=j+1}^{2k}\sum_{\ell=1}^{f-1}X^{(\ell)}_{b_{2i-1}b_{2i}}\right).
\end{align}
If we take $f=\lceil 4k(\fa+1)/\fc\rceil$, then on the event $\cal A$, the second term on the righthand side of \eqref{e:decompsum} is bounded,
\begin{align*}
&\left|\sum_{j=1}^{2k}\left(\prod_{i=1}^{j-1}X_{b_{2i-1}b_{2i}} \right)X^{(\infty)}_{b_{2j-1}b_{2j}}\left( \prod_{i=j+1}^{2k}\sum_{\ell=1}^{f-1}X^{(\ell)}_{b_{2i-1}b_{2i}}\right)\right|\\
=&\left|\sum_{j=1}^{2k}\left(\prod_{i=1}^{j-1}X_{b_{2i-1}b_{2i}} \right)X^{(\infty)}_{b_{2j-1}b_{2j}} \prod_{i=j+1}^{2k}\left(X_{b_{2i-1}b_{2i}}-X_{b_{2i-1}b_{2i}}^{(\infty)}\right)\right|\\
\leq & \frac{4k}{\eta^{2k}}\left(\frac{\psi}{\sqrt{N\eta}}\right)^f 
\leq4k \cY^{2k},
\end{align*}
where in the last inequality we used $\psi=N^{\fc}$ (as in \eqref{control}) and $\eta\geq\psi^4/N$, since $z\in \cal D_\kappa$ (as in \eqref{e:dom}).
%
%
%
This combining with \eqref{eqn:comperror} leads to
\begin{align}\label{e:cuttail}
\bE\left[|\cal Z|^{2k}\right]=\sum_{1\leq \ell_1,\cdots, \ell_{2k}\leq f-1}\sum_{\bmb} q_{b_1}q_{b_2}\cdots q_{b_{4k}}\bE\left[X^{(\ell_1)}_{b_1b_2}X^{(\ell_2)}_{b_3b_4}\cdots X^{(\ell_{2k})}_{b_{4k-1}b_{4k}}1_{\cal A}\right]+ O\left(4k\cY^{2k}\right).
\end{align}
By the Cauchy-Schwarz inequality, we have
\begin{align}\label{e:leading}
\left|\bE\left[X^{(\ell_1)}_{b_1b_2}\cdots X^{(\ell_{2k})}_{b_{4k-1}b_{4k}}1_{\cal A}\right]
\right|\leq \left|\bE\left[X^{(\ell_1)}_{b_1b_2}\cdots X^{(\ell_{2k})}_{b_{4k-1}b_{4k}}\right]
\right|+\bE \left[\left|X^{(\ell_1)}_{b_1b_2}\cdots X^{(\ell_{2k})}_{b_{4k-1}b_{4k}}\right|^2\right]\bP[\cal A^c]
\end{align}
In the following we bound the first term on the righthand side of \eqref{e:leading}, the second term can be treated in exactly the same way. By our definition of $X_{b_{2i-1}b_{2i}}^{(\ell_i)}$'s, we have
\begin{align}\label{e:Xexpand}
\bE\left[X_{b_1b_2}^{(\ell_1)}\cdots X_{b_{4k-1}b_{4k}}^{(\ell_{2k})}\right]
=\bE\left[\sum_{\bma:\bmb\subset\bma} \prod_{i=1}^{2k} \tilde{g}_{a^i_1}\tilde{\cal E}_{a^i_1a^i_2}\tilde{g}_{a^i_2}\cdots  \tilde{\cal E}_{a^i_{\ell_i} a^i_{\ell_{i}+1}}\tilde{g}_{a^i_{\ell_i+1}}\right],
\end{align}
where $\bma$ represents arrays $a^{i}_j\in \bT =\{b_1,b_2,\cdots, b_{4k}\}$, with indices $1\leq i\leq 2k$ and $1\leq j\leq \ell_i+1$; the above sum is over all the possible arrays $\bma$ containing $\bmb$, denoted by $\bmb\subset\bma$, in the sense that $a^{i}_1=b_{2i-1}$ and $a^{i}_{\ell_{i}+1}=b_{2i}$ for $1\leq i\leq 2k$. For the tilde notation, $\tilde{g}_{a_{j}^i}:=g_{a_{j}^i}$ and $\tilde{\cal E}_{a_j^i a_{j+1}^i}:=\cal E_{a_j^i a_{j+1}^i}$ for $1\leq i\leq k$, and $\tilde{g}_{a_{j}^i}:=g_{a_{j}^i}^*$ and $\tilde{\cal E}_{a_j^i a_{j+1}^i}:=\cal E_{a_j^i a_{j+1}^i}^*$ for  $k+1\leq i\leq 2k$.

Since by our definition $g_i$ are all deterministic, we can separate the deterministic part and the random part of \eqref{e:Xexpand}:
\begin{align}\label{e:Xexpand1}
\left|\bE\left[\sum_{\bma:\bmb\subset \bma} \prod_{i=1}^{2k} \tilde{g}_{a^i_1}\tilde{\cal E}_{a^i_1a^i_2}\tilde{g}_{a^i_2}\cdots  \tilde{\cal E}_{a^i_{\ell_i} a^i_{\ell_{i}+1}}\tilde{g}_{a^i_{\ell_i+1}}\right]\right|\leq &\sum_{\bma:\bmb\subset\bma} \prod_{i=1}^{2k}\prod_{j=1}^{\ell_i+1}|g_{a_j^i}|
\left|\bE\left[  \prod_{i=1}^{2k}\prod_{j=1}^{\ell_i} \tilde{\cal E}_{a_j^ia_{j+1}^i}
\right]\right|.
\end{align}
For the control of the expectation of the product of $\tilde{\cal E}_{ij}$, we have the following proposition, whose proof we postpone to the next section.

\begin{proposition}\label{p:weakCor}
For any indices $b_1, b_2, \cdots, b_{2\ell}\in \bT$, we have
\begin{align}
\label{E}\left|\bE\left[\tilde{\cal E}_{b_1b_2}\tilde{\cal E}_{b_3b_4}\cdots \tilde{\cal E}_{b_{2\ell-1}b_{2\ell}}\right]\right|\lesssim_\ell \frac{(\psi\log N)^{\ell}(t+\eta)^\ell}{(N\eta)^{\ell/2}}\chi(b_1,b_2,\cdots, b_{2\ell}).
\end{align}
where $\chi$ is an indicator function such that $\chi=1$ if any number in the array $(b_1,b_2,\cdots, b_{2\ell})$ occurs even number of times, otherwise $\chi=0$. 
\end{proposition}

Notice that $\chi((a_j^i,a_{j+1}^i)_{1\leq i\leq 2k, 1\leq j\leq \ell_i})=\chi((a^1_i,a^{\ell_i+1}_i)_{1\leq i\leq 2k})=\chi(\bmb)$. With Proposition \ref{p:weakCor}, we can bound \eqref{e:Xexpand1} as
\begin{align}\label{E2}
\bE\left[\sum_{\bma:\bmb\subset \bma} \prod_{i=1}^{2k} \tilde{g}_{a^i_1}\tilde{\cal E}_{a^i_1a^i_2}\tilde{g}_{a^i_2}\cdots  \tilde{\cal E}_{a^i_{\ell_i} a^i_{\ell_{i}+1}}\tilde{g}_{a^i_{\ell_i+1}}\right]
\lesssim_k\sum_{\bma:\bmb \subset \bma}\left(\frac{\psi(t+\eta)\log N}{(N\eta)^{1/2}}\right)^{\sum \ell_i}\chi(\bmb) \prod_{i=1}^{2k}\prod_{j=1}^{\ell_i+1}|g_{a_j^i}|.
\end{align}
where we used the fact that  $\sum \ell_i\leq 2k(f-1)\leq8k^2(\fa+1)/\frak{c}$, so the implicit constant depends only on $k$. 

Combining \eqref{e:Xexpand}, \eqref{e:Xexpand1} and \eqref{E2} together, 
\begin{align}\begin{split}\label{e:simplify}
 &\quad\sum_{\bmb} \left|q_{b_1}q_{b_2}\cdots q_{b_{4k}}\bE\left[X^{(\ell_1)}_{b_1b_2}X^{(\ell_2)}_{b_3b_4}\cdots X^{(\ell_{2k})}_{b_{4k-1}b_{4k}}\right]\right|\\
&=\sum_{\bmb} \left|q_{b_1}q_{b_2}\cdots q_{b_{4k}}\bE\left[\sum_{\bma:\bmb\subset\bma} \prod_{i=1}^{2k} \tilde{g}_{a^i_1}\tilde{\cal E}_{a^i_1a^i_2}\tilde{g}_{a^i_2}\cdots  \tilde{\cal E}_{a^i_{\ell_i} a^i_{\ell_{i}+1}}\tilde{g}_{a^i_{\ell_i+1}}\right]\right|\\
&\lesssim_k \left(\frac{\psi(t+\eta)\log N}{(N\eta)^{1/2}}\right)^{\sum \ell_i}\sum_{\bmb}\sum_{\bma:\bmb\subset\bma}\chi(\bmb)\prod_{i=1}^{2k}|q_{a_1^i}q_{a_{\ell_i+1}^i}|\prod_{i=1}^{2k}\prod_{j=1}^{\ell_i+1}|g_{a_j^i}|
\end{split}
\end{align}
Given $\bmb$, the sum $\sum_{\bma:\bma\subset \bmb}$ is  over all the possible arrays $\bma$ such that $a_j^i\in\bT=\{b_1,b_2,\cdots, b_{4k}\}$, for $1\leq i\leq 2k, 1\leq j \leq \ell_i+1$, and $a_1^i=b_{2i-1}$ and $a_{\ell_i+1}^i=b_{2i}$ for $1\leq i\leq 2k$. 
Since any array $\{a_j^i\}_{1\leq i\leq 2k, 1\leq j \leq \ell_i+1}$ induces a partition $\cal P$ of its index set $\{(i,j): 1\leq i\leq 2k, 1\leq j \leq \ell_i+1\}$, such that $(i,j)$ and $(i',j')$ are in the same block if and only if $a_j^i=a_{j'}^{i'}$. For any array $\bma$ with $\bmb\subset \bma$ and $\chi(\bmb)=1$ (as in Proposition \ref{p:weakCor}), we denote the frequency representation of the array $(b_1,b_2,\cdots, b_{4k})=(a_1^1,a_{\ell_1+1}^1, \cdots, a_{1}^{2k}, a_{\ell_{2k}+1}^{2k})$ as $\gamma_1^{d_1}\gamma_{2}^{d_2}\cdots\gamma_n^{d_n}$, where  $2\leq d_1,d_2, \cdots,d_n$ are all even, and $n=|\bT|$. Notice that  $\sum d_i=4k$ counts the total number. We also denote the frequency representation of $((a_j^i)_{1\leq j \leq \ell_i+1})_{1\leq i\leq 2k}$ as $\gamma_1^{d_1+r_1}\gamma_{2}^{d_2+r_2}\cdots\gamma_n^{d_n+r_n}$, where $r_i\geq 0$. Similarly, $\sum d_i+r_i=2k+\sum \ell_i$ counts the total number. 
We summarize here the relations between $d_i$, $r_i$ and $\ell_i$, which will be used later:
\begin{align*}
\sum d_i =4k, \quad 2k+\sum r_i=\sum \ell_i. 
\end{align*}

\begin{example}
If we take $k=3$, $\bmb=(1,2,2,3,4,5,3,5,2,1,2,4)$ and $\bma=((1,3,2,2)$;$(2,4,1,2,3)$;$(4,1,5)$; $(3,5)$;$(2,5,1,1)$; $(2,4,3,4,4,4,4))$, then $\bmb\subset \bma$. The partition $\cal P$ induced by $\bma$ is 
$\{\{(1,1),(2,3),(3,2),(5,3),(5,4)\}$,
$\{(1,3),(1,4), (2,1),(2,4), (5,1),(6,1)\}$,
$\{(1,2),(2,5),(4,1),(6,3)\}$,
$\{(2,2),(3,1),(6,2),(6,4),(6,5),(6,6),(6,7)\}$,
$\{(3,3),(4,2),(5,2)\} \}$. The frequency representations of $\bmb$ and $\bma$ are given by $1^22^44^23^25^2$ and $1^52^63^44^75^3$ respectively. $d_i$ and $r_i$ are given by $d_1=2, d_2=4, d_3=2,d_4=2, d_5=2$ and $r_1=3, r_2=2, r_3=2, r_4=5, r_5=1$. Since $d_1,d_2,\cdots, d_5$ are all even,  $\chi(\bmb)=1$.
\end{example}

Notice that the frequencies $d_i$ and $d_i+r_i$ are uniquely determined by the partition $\cal P$, in fact $d_i+r_i$ are the sizes of blocks of $\cal P$. Moreover since $\bmb$ is uniquely determined by $\bma$, first adding up terms corresponding to $\bma$ such that $\bmb\subset \bma$, and then summing over $\bmb$ is equivalent to first summing over arrays $\bma$ corresponding to the same partitions $\cal P$, which we denote by $\bma\sim \cal P$, and then summing  over different partitions with each block size at least two.
\begin{align}\begin{split}\label{e:simplify2}
\sum_{\bmb}\sum_{\bma:\bmb\subset\bma}\chi(\bmb)\prod_{i=1}^{2k}|q_{a_1^i}q_{a_{\ell_i+1}^i}|\prod_{i=1}^{2k}\prod_{j=1}^{\ell_i+1}|g_{a_j^i}|
&=\sum_{\cal P}\sum_{\bma \sim \cal P}\chi(\bmb)\prod_{i=1}^{2k}|q_{a_1^i}q_{a_{\ell_i+1}^i}|\prod_{i=1}^{2k}\prod_{j=1}^{\ell_i+1}|g_{a_j^i}|\\
&\leq\sum_{\cal P}\sum_{1\leq \gamma_1,\cdots,\gamma_n\leq N}
\prod_{i=1}^n|q_{\gamma_i}|^{d_i}|g_{\gamma_i}|^{d_i+r_i}\\
&\lesssim_k \sum_{\cal P} \prod_{i=1}^{n}\frac{\Im\left[\sum q_i^2 g_i\right]^{d_i/2}}{(t+\eta)^{r_i+d_i/2}}\\
&\leq \sum_{\cal P} \frac{\Im\left[\sum q_i^2g_i\right]^{2k}}{(t+\eta)^{\sum \ell_i}},
\end{split}\end{align}
where in the first inequality we use \eqref{e:highmoment} in Proposition \ref{p:sumqg} and $d_i\geq 2$, and for the last inequality we used $\sum d_i=4k$. 
Therefore by substituting \eqref{e:simplify2} into \eqref{e:simplify}, we have
\begin{align}\begin{split}\label{e:firsttermest}
&\sum_{\bmb} \left|q_{b_1}q_{b_2}\cdots q_{b_{4k}}\bE\left[X^{(\ell_1)}_{b_1b_2}X^{(\ell_2)}_{b_3b_4}\cdots X^{(\ell_{2k})}_{b_{4k-1}b_{4k}}\right]\right|
\lesssim_k \sum_{\cal P}\left(\frac{\psi (t+\eta)\log N}{(N\eta)^{1/2}}\right)^{\sum \ell_i}\frac{\Im\left[\sum q_i^2g_i\right]^{2k}}{(t+\eta)^{\sum \ell_i}}\\
\leq&
\sum_{\cal P}  \left(\frac{\psi\Im\left[\sum q_i^2g_i\right]\log N}{\sqrt{N\eta}}\right)^{2k} \left(\frac{\psi\log N}{\sqrt{N\eta}}\right)^{\sum r_i}
\lesssim_k  \cY^{2k},
\end{split}\end{align}
in the last inequality, we used that $\psi\log N/\sqrt{N\eta}\leq 1$ and the total number of partition is bounded by $(\sum\ell_i +2k)!$, which is a constant depending on $k$.

Following the same argument, one can check that
\begin{align}\label{e:secondtermest}
\sum_{\bmb} \left|q_{b_1}q_{b_2}\cdots q_{b_{4k}}\bE\left[\left|X^{(\ell_1)}_{b_1b_2}X^{(\ell_2)}_{b_3b_4}\cdots X^{(\ell_{2k})}_{b_{4k-1}b_{4k}}\right|^2\right]\right|\lesssim_k    \left(\frac{N\Im\left[\sum q_i^2g_i\right](\psi\log N)^2}{N\eta}\right)^{2k} 
\leq N^{2k}\cY^{2k}.
\end{align} 
Therefore, by combining \eqref{e:cuttail}, \eqref{e:leading}, \eqref{e:firsttermest} and \eqref{e:secondtermest}, it follows 
\begin{align*}
\bE[|\cZ|^{2k}]\lesssim_k \cY^{2k} + N^{2k}\cY^{2k}\bP(\cA^{c})\lesssim_k \cY^{2k}.
\end{align*}
This finishes the proof for isotropic law Theorem \ref{t:isolaw}.
\end{proof}

The following is an easy corollary of Theorem \ref{t:isolaw}:
\begin{corollary}
Under Assumption \ref{a:boundImm} and \ref{a:boundEv}, for any $\eta_*\ll t\ll r$, $0<\kappa<1$, 
we have that 
 \begin{align}
  \left|\langle \bmq, G(t,z)\bmq\rangle - \mfct(z)\right|\leq \frac{1}{N^\frak{b}}+\frac{\psi^2}{\sqrt{N\eta}},
 \end{align}
  uniformly for any $z\in \cal D_\kappa$, with overwhelming probability, provided $N$ is large enough.
\end{corollary}
\begin{proof}
By Assumption \ref{a:boundEv}, we have
\begin{align*}
\left|\sum_{i=1}^{N}\frac{\langle \bmu_i(0), \bmq\rangle^2}{\lambda_i(0)-z}-\frac{1}{N}\sum_{i=1}^N\frac{1}{\lambda_i(0)-z}\right|\leq \frac{1}{N^\frak{b}},
\end{align*}
uniformly for any $z\in \{E+\i \eta: E\in[E_0-r, E_0+r], \eta_*\leq \eta\leq r\}$. We denote $\tilde z=\tilde E+\i\tilde\eta := z+t\mfct(z)$. From Proposition \ref{p:mfc}, we know that for any $z\in \cal D_\kappa$, $\Im[z+t\mfct(z)]\gtrsim t+\eta\gg \eta_*$ and $|t\mfct(z)|\lesssim t\log N\ll \kappa r$ provided $N$ is large enough. Therefore, we have that $\tilde z\in \{E+\i \eta: E\in[E_0-r, E_0+r], \eta_*\leq \eta\leq 1\}$. As a consequence,
\begin{align*}
\Im\left[\sum_{i=1}^{N}\frac{\langle \bmu_i(0), \bmq\rangle^2}{\lambda_i(0)-z-t\mfct(z)}\right]
=\Im\left[\sum_{i=1}^{N}\frac{\langle \bmu_i(0), \bmq\rangle^2}{\lambda_i(0)-\tilde z}\right]\leq \frac{1}{N^\fb}+\Im\left[\frac{1}{N}\sum_{i=1}^{N}\frac{1}{\lambda_i(0)-\tilde z}\right]\lesssim 1.
\end{align*}
Combining with Theorem \ref{t:isolaw}, it follows that 
 \begin{align*}
  \left|\langle \bmq, G(t,z)\bmq\rangle - \sum_{i=1}^{N}\frac{\langle \bmu_i(0), \bmq\rangle^2}{\lambda_i(0)-z-t\mfct(z)}\right|\leq \frac{\psi^2}{\sqrt{N\eta}}.
 \end{align*}
Therefore with overwhelming probability we have
\begin{align*}
 &\left|\langle \bmq, G(t,z)\bmq\rangle - \mfct(z)\right|
 \leq \left|\langle \bmq, G(t,z)\bmq\rangle - \sum_{i=1}^{N}\frac{\langle \bmu_i(0), \bmq\rangle^2}{\lambda_i(0)-z-t\mfct(z)}\right|\\
 +&\left|\sum_{i=1}^{N}\frac{\langle \bmu_i(0), \bmq\rangle^2}{\lambda_i(0)-\tilde z}-\frac{1}{N}\sum_{i=1}^N\frac{1}{\lambda_i(0)-\tilde z}\right|
 \leq \frac{1}{N^\frak{b}}+\frac{\psi^2}{\sqrt{N\eta}}
\end{align*}
uniformly for any $z\in \cal D_\kappa$.
\end{proof}

We take the event $A_1$ of trajectories $(\bm \la(t))_{0\leq t\leq r}$ such that:  
\begin{enumerate}
\rm\item Eigenvalue rigidity holds:  $\sup_{t_0\leq s\leq \ell/N}|m_s(z)-m_{{\rm{fc}},s}(z)|\leq \psi (N\eta)^{-1}$  uniformly for $z\in \cal D_\kappa$; and $ \sup_{t_0\leq s\leq t_0+\ell/N}|\lambda_{i}(s)-\gamma_{i}(s)|\leq\psi N^{-1}$ uniformly for indices $i$ such that $\gamma_i(s)\in I_r^{\kappa}(E_0)$.
\rm\item When we conditioning on any trajectory $\bm \lambda\in A$, with overwhelming probability, the following holds 
\begin{align*}
\sup_{t_0\leq s\leq t_0+\ell/N}|\langle \bmq, G(t,z)\bmq\rangle-\mfct(z)|\leq \frac{1}{N^{\frak{b}}}+\frac{\psi}{\sqrt{N\eta}}
\end{align*}
uniformly for $z\in \cal D_\kappa$.
\end{enumerate}
As a consequence of Theorem \ref{t:rig} and \ref{t:isolaw}, and notice we can take the parameter $\fc$ (as in \eqref{control}) arbitrarily small, the event $A_1$ holds with overwhelming probability.

\subsection{Auxiliary results}
\begin{proposition}\label{p:sumqg}
 The initial matrix $H_0=\diag\{\lambda_1(0),\cdots, \lambda_N(0)\}$ satisfies Assumption \ref{a:boundImm}. Fix $\kappa>0$. Then for any $k\geq 2$ and $m\geq 0$, we have
\begin{align}\label{e:highmoment}
\sum_{i=1}^{N}q_i^{k}|g_i(t,z)|^{k+m}\lesssim_k \frac{\Im\left[\sum q_i^2 g_i\right]^{k/2}}{(t+\eta)^{k/2+m}}, 
\end{align}
and for any $m\geq 0$, we have
\begin{align}\label{e:1moment}
\sum_{i=1}^{N}|q_i||g_i(t,z)|^{2+m}\lesssim \frac{N^{1/2}\Im\left[\sum q_i^2g_i\right]^{1/2}}{(t+\eta)^{1+m}}, 
\end{align}
and
\begin{align}\label{e:0moment}
\sum_{i=1}^{N}|g_i(t,z)|^{1+m}\lesssim \frac{N\log N}{(t+\eta)^m}, 
\end{align}
uniformly for any $z\in \cal D_\kappa$, where $g_i$ are as in \eqref{def:gi}. 
\end{proposition}

\begin{proof}
We denote $\tilde E+\i\tilde\eta := z+t\mfct(z)$. From Proposition \ref{p:mfc}, $\tilde \eta=\Im[z+t\mfct(z)]\gtrsim (\eta+t)$, which gives us a rough bound for $g_i(t,z)$ :
\begin{align}\label{e:trivialb}
|g_i(t,z)|\lesssim (t+\eta)^{-1}.
\end{align}
With the trivial bound \eqref{e:trivialb}, \eqref{e:highmoment} and \eqref{e:1moment} are reduced to the case $m=0$. 
For \eqref{e:highmoment}, we have the basic inequality $\sum x_i^k\leq \left(\sum x_i^{2}\right)^{k/2}$ if $k\geq 2$. Therefore,
\begin{align*}
\sum_{i=1}^N q_i^{k}|g_i(t,z)|^{k}\leq \left(\sum_{i=1}^N q_i^{2}|g_i(t,z)|^{2}\right)^{k/2}
=\left(\frac{\Im\left[\sum q_i^2 g_i\right]}{\Im[z+t\mfct(z)]}\right)^{k/2}\lesssim_k \frac{\Im\left[\sum q_i^2 g_i\right]^{k/2}}{(t+\eta)^{k/2}}.
\end{align*}
For \eqref{e:1moment}, by Cauchy's inequality
\begin{align*}
\sum_{i=1}^{N}|q_i||g_i(t,z)|^{2}
&\leq \left(\sum_{i=1}^{N}|g_i(t,z)|^{2}\right)^{1/2}\left(\sum_{i=1}^{N}|q_i|^2|g_i(t,z)|^{2}\right)^{1/2}\\
&=\left(\frac{\Im\left[\sum g_i\right]}{\Im[z+t\mfct(z)]}\right)^{1/2}\left(\frac{\Im\left[\sum q_i^2 g_i\right]}{\Im[z+t\mfct(z)]}\right)^{1/2}\\
&=\frac{N^{1/2}\Im[\mfct(z)]^{1/2}\Im\left[\sum q_i^2 g_i\right]^{1/2}}{t+\eta}\lesssim \frac{N^{1/2}\Im\left[\sum q_i^2g_i\right]^{1/2}}{t+\eta},
\end{align*}
where we used $\Im[\mfct(z)]\leq C$ from \eqref{e:dbound}.
Finally, \eqref{e:0moment} in the case $m=0$ is the same as \eqref{e:sumgibound}.
\end{proof}

\begin{proof}[Proof of Proposition \ref{p:weakCor}] Recall the decomposition $\cal E=\cal E^{(1)}+\cal E^{(2)}+\cal E^{(3)}$ from \eqref{defcE}. If we condition on the submatrix $[W]_{i,j\notin \bT} $, $\cal E^{(1)}$ is diagonal and non-random, $\cal E^{(2)}$ depends  on $[W]_{i,j\in \bT}$, and $\cal E^{(3)}$ depends on $W_{i\notin \bT,j\in \bT}$, so they are independent. \eqref{E} can be decomposed into the following three estimates: with overwhelming probability
\begin{align}
&\label{E1}\left|\bE_\bT\left[\tilde{\cal E}_{b_1b_2}^{(1)}\tilde{\cal E}_{b_3b_4}^{(1)}\cdots \tilde{\cal E}_{b_{2\ell-1}b_{2\ell}}^{(1)}\right]\right|\leq \left(\frac{\psi t}{N\eta}\right)^\ell\chi(b_1,b_2,\cdots, b_{2\ell}),\\
&\label{E2}\left|\bE_{\bT}\left[\tilde{\cal E}_{b_1b_2}^{(2)}\tilde{\cal E}_{b_3b_4}^{(2)}\cdots \tilde{\cal E}_{b_{2\ell-1}b_{2\ell}}^{(2)}\right]\right|\lesssim_\ell \left(\frac{t}{N}\right)^{\ell/2}\chi(b_1,b_2,\cdots, b_{2\ell}),\\
&\label{E3}\left|\bE_\bT\left[\tilde{\cal E}_{b_1b_2}^{(3)}\tilde{\cal E}_{b_3b_4}^{(3)}\cdots \tilde{\cal E}_{b_{2\ell-1}b_{2\ell}}^{(3)}\right]\right|\lesssim_\ell \left(\frac{\psi t\log N}{\sqrt{N\eta}}\right)^\ell\chi(b_1,b_2,\cdots, b_{2\ell}),
\end{align}
where $\bE_{\bT}$ is the expectation with respect to rows and columns of $W$ indexed by $\bT$.

For \eqref{E1}, since $\cE^{(1)}$ is diagonal and by \eqref{eqn:locallaw} in Theorem \ref{t:rig}, with overwhelming probability, $t|m_t^{(\bT)}-\mfct|\leq  \psi t/(N\eta)$, we have
\begin{align*}
\left|\bE_\bT\left[\tilde{\cal E}_{b_1b_2}^{(1)}\tilde{\cal E}_{b_3b_4}^{(1)}\cdots \tilde{\cal E}_{b_{2\ell-1}b_{2\ell}}^{(1)}\right]\right|\leq \left(\frac{\psi t}{N\eta}\right)^\ell\prod_{i=1}^{\ell}\delta_{b_{2i-1}b_{2i}}
\leq \left(\frac{\psi t}{N\eta}\right)^\ell\chi(b_1,b_2,\cdots, b_{2\ell}).
\end{align*}
For  \eqref{E2}, it is a product of normal variables, which does not vanish only if each variable occurs even number of times. Thus \eqref{E2} follows, and the implicit constant is from the moment of normal variables, and can be bounded by $(2\ell-1)!!$.

In the following we prove \eqref{E3}. The entries of $\cal E^{(3)}$ are given by
\begin{align*}
\cal E_{b_{2i-1}b_{2i}}^{(3)}=t\sum_{\beta_{2i-1},\beta_{2i}\notin \bT}\left(w_{b_{2i-1}\beta_{2i-1}}w_{b_{2i}\beta_{2i}}-\frac{\delta_{b_{2i-1}b_{2i}}\delta_{\beta_{2i-1}\beta_{2i}}}{N}\right)G_{\beta_{2i-1}\beta_{2i}}^{(\bT)}.
\end{align*}
Therefore, the lefthand side of $\eqref{E3}$ is bounded by 
\begin{align*}
t^\ell\sum_{\beta_1\beta_2,\cdots \beta_{2\ell}\notin\bT}\left|\bE_{\bT}\left[\prod_{i=1}^{\ell}\left(w_{b_{2i-1}\beta_{2i-1}}w_{b_{2i}\beta_{2i}}-\frac{\delta_{b_{2i-1}b_{2i}}\delta_{\beta_{2i-1}\beta_{2i}}}{N}\right)\right]\right|\left|G_{\beta_1\beta_2}^{(\bT)}\cdots G_{\beta_{2\ell-1}\beta_{2\ell}}^{(\bT)}\right|
\end{align*}
For each monomial of resolvent entries $G_{\beta_1\beta_2}^{(\bT)}\cdots G_{\beta_{2\ell-1}\beta_{2\ell}}^{(\bT)}$, we associate it with a labeled graph $\cal G$ in the following procedure: We denote the frequency representation of the array $(\beta_1,\beta_2,\cdots, \beta_{2\ell})$ as $\gamma_1^{d_1}\gamma_2^{d_2}\cdots \gamma_v^{d_v}$, where $d_i\geq1$ is the multiplicity of $\gamma_i$, and $v=|\{\beta_1,\beta_2,\cdots, \beta_{2\ell}\}|$.
We construct the labeled graph $\cal G$ with vertex set $\{\gamma_1,\gamma_2,\cdots, \gamma_v\}$ and $\ell$ edges $(\beta_{2i-1},\beta_{2i})$ for $1\leq i\leq \ell$ (if $\beta_{2i-1}=\beta_{2i}$, the edge $(\beta_{2i-1},\beta_{2i})$ is a self-loop).
We denote $s$ the number of self-loops in $\cal G$. For any vertex $\gamma_i\in \cal G$, its degree is given by $d_i$, where self-loop adds two to the degree.
It is easy to see that \eqref{E3} follows from combining the following two estimates:
\begin{align}\label{e:gbound}
\left|\bE_{\bT}\left[\prod_{i=1}^{\ell}\left(w_{b_{2i-1}\beta_{2i-1}}w_{b_{2i}\beta_{2i}}-\frac{\delta_{b_{2i-1}b_{2i}}\delta_{\beta_{2i-1}\beta_{2i}}}{N}\right)\right]\right|\lesssim_\ell \frac{1}{N^\ell} \rho(\cal G)\chi(b_1,b_2,\cdots,b_{2\ell}),
\end{align}
where the implicit constant is from the moment of normal variables, and can be bounded by $(2\ell-1)!!$; and with overwhelming probability, uniformly for any $z\in \cal D_\kappa$,
\begin{align}\label{e:resoprod}
\sum_{\beta_1\beta_2,\cdots \beta_{2\ell}\notin\bT}\left|G_{\beta_1\beta_2}^{(\bT)}\cdots G_{\beta_{2\ell-1}\beta_{2\ell}}^{(\bT)}\right|\rho(\cal G)\lesssim_\ell \frac{(\psi\log N)^{\ell} N^{\ell/2}}{\eta^{\ell/2}},
\end{align}
where $\rho(\cal G)$ is an indicator function, which equals one if each vertex of $\cal G$ is incident to two different edges, otherwise it is zero. For any graph $\cal G$ with $\rho(\cal G)=1$, we count the total number of edge-vertex pairs, such that the vertex is incident to the edge: each self-loop contributes to $1$, and each non self-loop contributes to two, so the total number is $s+2(\ell-s)$; since each vertex of $\cal G$ is incident to at least two different edges, the total number is at least $2v$. Therefore, we have the following relation between $v$, $s$ and $\ell$:
\begin{align}\label{e:rela}
2v\leq 2(\ell-s)+s=2\ell-s.
\end{align}

For the first bound \eqref{e:gbound}, we denote the set $B=\{(b_j,\beta_j)\}_{1\leq j\leq 2\ell}$. Then the product in \eqref{e:gbound} can be rewritten as 
\begin{align*}
\prod_{i=1}^{\ell}\left(w_{b_{2i-1}\beta_{2i-1}}w_{b_{2i}\beta_{2i}}-\frac{\delta_{b_{2i-1}b_{2i}}\delta_{\beta_{2i-1}\beta_{2i}}}{N}\right)
=\prod_{(b,\beta)\in B} w_{b\beta}^{e_1(b,\beta)}(w_{b\beta}^2-1/N)^{e_2(b,\beta)},
\end{align*}
where $e_1(b,\beta)=|\{1\leq i\leq \ell:\text {exact one of $(b_{2i-1},\beta_{2i-1}), (b_{2i},\beta_{2i})$ is $(b,\beta)$}\}|$ and $e_2(b,\beta)=|\{1\leq i\leq \ell:(b_{2i-1},\beta_{2i-1})=(b_{2i},\beta_{2i})=(b,\beta)\}|$. Since for $(b,\beta)\in B$, $w_{b\beta}$ are independent normal random variables, \eqref{e:gbound} does not vanish only if $e_1(b,\beta)$ is even and $e_1(b,\beta)+e_2(b,\beta)\geq 2$ for any $(b,\beta)\in B$, which implies $\rho(\cal G)\chi(b_1,b_2,\cdots,b_{2\ell})=1$. Therefore, we have
\begin{align*}
\bE_{\bT}\left[\prod_{(b,\beta)\in B} w_{b\beta}^{e_1(b,\beta)}(w_{b\beta}^2-1/N)^{e_2(b,\beta)}\right]
&\lesssim_{\ell}\frac{1}{N^{\sum_{(b,\beta)\in B}e_1(b,\beta)/2}N^{\sum_{(b,\beta)\in B}e_2(b,\beta)}}\rho(\cal G)\chi(b_1,b_2,\cdots,b_{2\ell})\\
&=\frac{1}{N^{\ell}}\rho(\cal G)\chi(b_1,b_2,\cdots,b_{2\ell}),
\end{align*}
and \eqref{e:gbound} follows.

For the second bound \eqref{e:resoprod}, by Proposition \ref{p:resoentry}, with overwhelming probability we have
\begin{align*}
|G_{\beta_{2i-1}\beta_{2i}}^{(\bT)}|\leq 
\left\{
\begin{array}{cc}
\psi (|g_{\beta_{2i-1}}||g_{\beta_{2i}}|)^{1/2}, & \beta_{2i-1}=\beta_{2i},\\
\psi (|g_{\beta_{2i-1}}||g_{\beta_{2i}}|)^{1/2}/\sqrt{N\eta} , &\beta_{2i-1}\neq \beta_{2i}.
\end{array}
\right.
\end{align*}
In terms of the graph $\cal G$, the first bound corresponds to self-loops, and the second bound corresponds to non self-loop edges. In the graph $\cal G$, there are $s$ self-loops and $\ell-s$ non self-loop edges. The product of resolvent entries can be bounded as
\begin{align*}
\left|G_{\beta_1\beta_2}^{(\bT)}\cdots G_{\beta_{2\ell-1}\beta_{2\ell}}^{(\bT)}\right|\rho(\cal G)
\leq &\frac{ \psi^{\ell}}{\sqrt{N\eta}^{\ell-s}}\prod_{i=1}^{v}|g_{\gamma_i}|^{\frac{d_i}{2}}\rho(\cal G),
\end{align*}
with overwhelming probability. Notice that $\rho(\cG)=1$ implies that $d_i\geq 2$. The index set $(\beta_1,\beta_2,\cdots, \beta_{2\ell})$ induces a partition $\cal P$ on the set $\{1,2,\cdots, 2\ell\}$ such that $i$ and $j$ are in the same block if and only if $\beta_i=\beta_j$. If two index sets induce the same partition, they correspond to isomorphic graphs (when we forget the labeling). Therefore, for \eqref{e:resoprod}, we can first sum over the index sets corresponding to the same partition and then sum over different partitions:
\begin{align*}
&\sum_{\beta_1,\cdots,\beta_{2\ell}\notin \bT}\left|G_{\beta_1\beta_2}^{(\bT)}\cdots G_{\beta_{2\ell-1}\beta_{2\ell}}^{(\bT)}\right|\rho(\cal G)
=\sum_{\cal P}\sum_{(\beta_1,\cdots,\beta_{2\ell})\sim \cal P}\left|G_{\beta_1\beta_2}^{(\bT)}\cdots G_{\beta_{2\ell-1}\beta_{2\ell}}^{(\bT)}\right|\rho(\cal G)\\
&\leq \sum_{\cal P}\frac{ \psi^{\ell}}{\sqrt{N\eta}^{\ell-s}}\sum_{\gamma_1,\cdots,\gamma_v\notin \bT}\prod_{i=1}^v|g_{\gamma_i}|^{\frac{d_i}{2}}\rho(\cal G)
\lesssim_\ell \sum_{\cal P}\frac{ \psi^{\ell}}{\sqrt{N\eta}^{\ell-s}}\prod_{i=1}^{v}\frac{N\eta\log N}{\eta^{\frac{d_i}{2}}}\\
&\leq\psi^{\ell}\sum_{\cal P}\frac{(N\eta\log N)^{v}}{ (N\eta)^{(\ell-s)/2}\eta^\ell}
\leq \psi^{\ell}\sum_{\cal P}\frac{(N\eta\log N)^{\ell-s/2}}{ (N\eta)^{(\ell-s)/2}\eta^{\ell}}
\lesssim_\ell \frac{(\psi\log N)^{\ell} N^{\ell/2}}{\eta^{\ell/2}}
\end{align*}
where the second inequality follows from \eqref{e:0moment}, in the third inequality, we used $\sum_i d_i=2\ell$, for the second to last inequality, we used the bound $v\leq \ell-s/2$ from \eqref{e:rela}, and in the last inequality, we bounded the total number of different partitions by $(2\ell)!$.
\end{proof}

\section{Short Time Relaxation}\label{relaxation}

The Dyson Brownian motion \eqref{DBM} induces the following two dynamics on eigenvalues and eigenvectors, 
\begin{align}
\rd\la_k(t)&=\frac{\rd b_{kk}(t)}{\sqrt{N}}+\left(\frac{1}{N}\sum_{\ell\neq k}\frac{1}{\la_k(t)-\la_\ell(t)}\right)\rd t\label{eqn:eiflow},\\
\rd \bmu_k(t)&=\frac{1}{\sqrt{N}}\sum_{\ell\neq k}\frac{\rd b_{k\ell}(t)}{\lambda_k(t)-\lambda_\ell(t)}\bmu_\ell(t)
-\frac{1}{2N}\sum_{\ell\neq k}\frac{\rd t}{(\la_k-\la_\ell)^2}\bmu_k(t)\label{eqn:evflow},
\end{align}
where $B_t=(b(t))_{1\leq i,j\leq N}$ is symmetric with $(b_{ij}(t))_{1\leq i\leq j\leq N}$ a family of independent Brownian motions with variance $(1+\delta_{ij})t$. Following the convention of \cite[Definition 2.2]{BoYaQUE}, we call them Dyson Brownian motion for \eqref{eqn:eiflow} and Dyson vector flow for \eqref{eqn:evflow}.

In order to study the Dyson vector flow, the {\it moment flow} was introduced in \cite[Section 3.1]{BoYaQUE}, where the observables are the moments of projections of the eigenvectors onto a given direction. For any unit vector $\bmq\in \bR^N$, and any index $1\leq k\leq N$, define: $z_k(t)=\sqrt{N}\langle \bmq, \bmu_k(t)\rangle$, where with the $\sqrt{N}$ normalization, the typical size of $z_k$ is of order $1$. The normalized test functions are 
\begin{align}
&{\rQ_{ t}}_{ i_1,\dots,i_m}^{j_1,\dots,j_m}= \prod_{\ell=1}^mz_{i_\ell}^{2j_\ell} \label{eqn:rescaleSym}
\prod_{\ell=1}^m a(2j_\ell)^{-1}\ \mbox{where}\ a(2j)=(2j-1)!!,
\end{align}
These indices, $\{(i_1,  j_1),  \dots, (i_m , j_m) \}$ with distinct $i_k$'s and positive $j_k$'s can be encoded as a particle configuration $\bm \eta=(\eta_1,\eta_2,\cdots,\eta_N)$ on $\qq{1, N}$ such that $\eta_{i_k}=j_k$ for $1\leq k\leq m$ and $\eta_p=0$ if $p\notin\{i_1,i_2,\cdots, i_m\}$. The total number of particles is $\cN(\boeta):=\sum \eta_\ell=\sum j_k$. We denote the particles in non-decreasing order by $x_1(\bm\eta) \leq x_2(\bm\eta)\leq \cdots\leq x_{\cN(\boeta)}(\boeta)$. If the context is clear we will drop the dependence on $\bm\eta$. We also say the support of $\boeta$ is $\{i_1,i_2,\cdots, i_m\}$.  It is easy to see that this is a bijection between test functions ${\rQ_{ t}}_{ i_1,\dots,i_m}^{j_1,\dots,j_m}$ and particle configurations. We define $\boeta^{i j}$ to be the configuration by moving one particle from $i$ to $j$. For any pair of $n$ particle configurations $\boeta$: $1\leq x_1\leq x_2\leq\cdots \leq x_n\leq N$ and $\boxi$: $1\leq y_1\leq y_2\leq \cdots \leq y_n \leq N$, we  define the following distance: 
\begin{equation}\label{dis}
d(\boeta, \boxi) = \sum_{\alpha=1}^n |x_\alpha-y_\alpha|.
\end{equation}

We condition on the trajectory of the eigenvalues, and define 
\beq\label{feq}
  f^{H_0}_{\bla, t} (\boeta)  
  =\bE^{H_0} ( {\rQ_{ t}}_{i_1,\dots,i_m}^{j_1,\dots,j_m} (t)\mid \bla ),
\bEq
where $\boeta$ is  the  configuration  $ \{(i_1,  j_1),  \dots, (i_m , j_m ) \}$. Here $\bla$ denotes the whole path of eigenvalues for $0\leq t\leq 1$.
 The dependence in the initial matrix $H_0$ will often be omitted so that we write  $f_{t}=f^{H_0}_{\bla, t}$. We will call $f_{t}$ the {\it eigenvector moment flow}, which is governed by the following generator $\mathscr{B}(t)$ \cite[Theorem 3.1]{BoYaQUE}:

\begin{theorem}\label{t:emf}[Eigenvector moment flow] Let $\bmq\in\mathbb{R}^N$ be a unit vector, $z_k=\sqrt{N}\langle\bmq, \bmu_k (t) \rangle$ and $c_{ij}(t)= (\lambda_i(t) - \lambda_j(t))^{-2}/N$.
Suppose that $ f_{t} (\boeta)$ is given by \eqref{feq} where   
$\boeta$ denote  the  configuration  $ \{(i_1,  j_1),  \dots, (i_m , j_m ) \}$. \nc Then 
$f_{ t}$
satisfies the equation  
\begin{align}
\label{ve}
&\partial_t f_{t} =  \mathscr{B}(t)  f_{t},\\
\label{momentFlotSym}&\mathscr{B}(t)  f_t(\boeta) = \sum_{i \neq j} c_{ij}(t) 2 \eta_i (1+ 2 \eta_j) \left(f_t(\boeta^{i, j})-f_t(\boeta)\right).
\end{align}
\end{theorem}
\noindent An important property of the eigenvector moment flow is the reversibility with respect to  a simple explicit equilibrium measure:
\beq\label{eqn:weight}
\pi(\boeta) = \prod_{p=1}^N \phi(\eta_p), \ \phi(k) =\prod_{i=1}^k\left(1-\frac{1}{2i}\right).
\bEq
And for any function $f$ on the configuration space, the Dirichlet form is given by 
\begin{align*}
\sum_{\bm\eta}\pi(\bm\eta)f(\bm\eta)\mathscr{B}(t)f(\bm \eta)=\sum_{\boeta}  \pi(\boeta) \sum_{i \neq j}  c_{ij}  \eta_i (1+ 2 \eta_j) \left(f(\boeta^{i j}) - f(\boeta)\right)^2.
\end{align*}

We are interested in the eigenvectors corresponding to eigenvalues on the interval $[E_0-r, E_0+r]$, and we only have local information of the initial matrix $H_0$. However, the operator $\mathscr{B}(t)$ has long range interactions. We fix a short range parameter $\ell$, and split $\mathscr{B}(t)$ into short-range part and long range part: $\mathscr{B}(t)=\mathscr{S}(t)+\mathscr{L}(t)$, with
\begin{align}
&(\mathscr{S} f_t)(\boeta)  =   \sum_{0<|j-k|\leq \ell }  c_{jk}(t) 2 \eta_j (1+ 2 \eta_k) \left(f_t(\boeta^{j k}) - f_t(\boeta)\right),\label{eqn:shortrange}\\
&(\mathscr{L}f_t) (\boeta) =    \sum_{|j-k|>\ell}  c_{jk}(t) 2 \eta_j (1+ 2 \eta_k ) \left(f_t(\boeta^{jk}) - f_t(\boeta)\right).\notag
\end{align}\label{eqn:longrange}
Notice that $\mathscr{S}$ and $\mathscr{L}$ are also reversible with respect to the measure $\pi$ (as in \eqref{eqn:weight}). We denote by $\rU_\mathscr{B}(s,t)$ ($\rU_\mathscr{S}(s,t)$ and $\rU_\mathscr{L}(s,t)$) the semigroup associated with $\mathscr{B}$ ($\mathscr{S}$ and $\mathscr{L}$) from time $s$ to $t$, i.e.
\begin{align*}
\del_t \rU_\mathscr{B}(s,t)=\mathscr{B}(t)\rU_\mathscr{B}(s,t).
\end{align*}

For any $\eta_*\ll t\ll r$, In the rest of this section, we fix time $t_0$ and the range parameter $\ell$, such that $\eta_*\ll t_0\leq t\leq t_0+\ell/N$, which we will choose later.  We will show that the effect of the long-range operator $\mathscr{L}(t)$ is negligible in the sense of $L^{\infty}$ norm, i.e. $f_t(\boeta)\approx {\rm U}_{\mathscr{S}}(t_0,t)f_{t_{0}}(\boeta)$; and the short-range operator $\mathscr{S}(t)$ satisfies certain finite speed of propagation estimate, and \eqref{eqn:shortrange} converges to equilibrium exponentially fast with rate $N$. As a consequence, $f_t(\boeta)\approx 1$ and Theorem \ref{t:normal} follows.

\subsection{Finite Speed of Propagation}

In this section, we fix some small parameter $0<\kappa<1$, and define the following efficient distance on $n$ particle configurations:
\begin{equation}\label{effdis}
 \tilde{d}(\boeta,\boxi)=\max_{1\leq \alpha\leq n}\#\{i\in\qq{1,N}: \gamma_{i}(t_0)\in I_{\kappa}^{r}(E_0), i\in \qq{x_\alpha, y_{\alpha}}\},
 \end{equation}
where $\boeta$: $1\leq x_1\leq x_2\leq\cdots \leq x_n\leq N$ and $\boxi$: $1\leq y_1\leq y_2\leq \cdots \leq y_n \leq N$, and $\gamma_i(t_0)$ are classical eigenvalue locations at time $t_0$ (as in \eqref{eqn:eigloc}).


In this section, we will condition on $\bm\lambda(t_0)=\bm\lambda$ for some ``good'' eigenvalue configuration $\bm\lambda$. We call an eigenvalue configuration $\bm\la$ good if we condition on $\bm \la(t_0)=\bm\la$, for $N$ large enough the following holds with overwhelming probability:
\begin{enumerate}
\item 
$\sup_{t_0\leq s\leq t}|m_s(z)-m_{{\rm{fc}},s}(z)|\leq \psi (N\eta)^{-1}$,
uniformly for any $z\in \cal D_\kappa$;
\item
$\sup_{t_0\leq s\leq t}|\lambda_{i}(s)-\gamma_{i}(s)|\leq\psi N^{-1}$,
uniformly for indices $i$ such that $\gamma_i(t)\in I_r^{\kappa}(E_0)$. 
\end{enumerate}
By Theorem \ref{t:rig}, combining with a simple continuity argument, $\bm\la(t_0)$ is a good eigenvalue configuration with overwhelming probablity.

\begin{lemma}\label{l:fspeed}
 Under the Assumption \ref{a:boundImm}, for any $\eta_*\ll t\ll r$,  we fix time $t_0$ and the range parameter $\ell$, such that $\eta_*\ll t_0\leq t\leq t_0+\ell/N\ll r$.  For any $n$ particle configurations $\bm \eta$: $1\leq x_1\leq x_2\leq\cdots \leq x_n\leq N$, and $\bm \xi$: $1\leq y_1\leq y_2\leq \cdots \leq y_n \leq N$, with $\tilde{d}(\boeta,\boxi)\geq \psi\ell/2$, then there exists a universal constant $c$, for $N$ large enough, the following holds with overwhelming probability:
 \begin{align}\label{eqn:fsp}
\sup_{t_0\leq s\leq t} \rU_{\mathscr S}(t_0,s)\delta_{\boeta} (\boxi)\leq e^{-c\psi},
 \end{align}
if we condition on $\bm\la(t_0)=\bm\la$, for any good eigenvalue configuration $\bm\la$.
\end{lemma}

 Thanks to the Markov property of the Dyson Brownian motion, we know that the conditioned law $(\bm\la(t))_{t\geq t_0}|\bm\lambda (t_0)=\bm\lambda$ is the same as Dyson Brownian motion starting at $\bm \la$. In the proof, we will neglect the conditioning in \eqref{eqn:fsp}, and simply think it as a Dyson eigenvalue flow starting at $\bm\la$. The proof of Lemma \ref{l:fspeed} consists of three steps. In the first step, we introduce some notations and define $X_s$, the weighted sum of $\rU_{\mathscr S}(t_0,s)\delta_{\boeta} (\boxi)$ over all configurations $\boxi$. In the second step, we prove Lemma \ref{l:fspeed} given the estimate \eqref{eqn:upperbound} of $X_s$. In the last step we prove the estimate \eqref{eqn:upperbound} by analyzing the stochastic equation of $X_s$. 
 
\begin{proof}
\emph{First Step:}
We denote $\nu=N/\ell$ and $r_s(\boeta,\boxi)= \rU_{\mathscr S}(t_0,s)\delta_{\boeta}(\boxi)$.  We define a family of cut-off functions $g_w$ parametrized by $w\in \bR$ by demanding that $\inf_x g_w(x) = 0$ and define $g_w'$ by  considering the following three cases: 
\begin{enumerate}
\item  $w \le E_0- (1-2\kappa)r$. Define 
\begin{align*}
g'_w (x)= \begin{cases} 1  &\mbox{if }   x\in I_{2\kappa}^r(E_0) \\ 
0 & \mbox{if }   x \notin I_{2\kappa}^r(E_0) \end{cases}    
\end{align*}
\item $ w\in I_{2\kappa}^r(E_0)$. Define 
\begin{align*}
g'_w (x)=   \begin{cases} 1  &\mbox{if }   x\ge w, \; x\in I_{2\kappa}^r(E_0)  \\ 
-1  &\mbox{if }   x<  w, \; x\in I_{2\kappa}^r(E_0)  \\ 
0 & \mbox{if }   x \not \in I_{2\kappa}^r(E_0) \end{cases}  
\end{align*}
\item  $w \ge E_0+(1-2\kappa)r$. Define 
\begin{align*}
g'_w (x)= \begin{cases} -1  &\mbox{if }   x\in I_{2\kappa}^r(E_0)  \\ 
0 & \mbox{if }   x \not \in I_{2\kappa}^r(E_0)\end{cases}    
\end{align*}
\end{enumerate}
It is easy to see that for any fixed $x$, as a function of $w$, $g_w'(x)$ is non-increasing. 
We take $\chi$ a smooth, nonnegative function, compactly supported on $[-1,1]$ with $\int\chi(x)\rd x=1$. We also define the smoothed version of $g_w$, 
$\varphi_i(x)=\int g_{\gamma_i(t_0)}(x-y)\nu\chi(\nu y)\rd y$.
Then $\varphi_i$ is smooth,
$\|\varphi_i'\|_\infty\leq1$ and $\|\varphi_i''\|_\infty\leq \nu$. Moreover, $\varphi_i(\gamma_i(t_0))\leq 1/\nu$, and $\varphi_i(x)$ all vanish for $x\leq E_0-(1-2\kappa)r-\ell/N$ or $x\geq E_0+(1-2\kappa)r+\ell/N$. From the monotonicity of $g_w'(x)$, for any $a\leq b$, we have $\lambda_a(t_0)\leq \lambda_b(t_0)$, so 
\beq\label{monopsi}
\varphi'_a(x)-\varphi'_b(x)  \geq 0.
\bEq

  We define the stopping time $\tau$, which is the first time $s\geq t_0$ such that either of the following fails: \rn{1}) $|m_s(z)-m_{\rm{fc},s}(z)|\leq \psi (N\eta)^{-1}$  uniformly for $z\in \cal D_\kappa$; \rn{2}) $ |\lambda_{i}(s)-\gamma_{i}(s)|\leq\psi N^{-1}$ uniformly for indices $i$ such that $\gamma_i(s)\in I_r^{\kappa}(E_0)$. By our assumption that $\bm\la(t_0)$ is a good configuration, we have that $\tau\geq t$ with overwhelming probability.
Recall the inverse Stieltjest transform, $\rho_{{\rm fc},s}(E)=\lim_{\eta\rightarrow 0} \Im[m_{{\rm fc},s}(E+\ri \eta)]/\pi$. By Proposition \ref{p:mfc}, the density $\rho_{{\rm fc},s}(E)$  is lower and upper bounded on $I_\kappa^r(E_0)$, on the scale $\eta\geq \psi^4/N$, and the same holds for and $N^{-1}\sum \delta_{\lambda_i(s)}$ by rigidity. Thus, there exists some universal constant $C$ such that for any $t_0\leq s\leq t$, and interval $I$ centered in $I_{\kappa}^r(E_0)$, with $|I|\geq \psi^4/N$, 
  \begin{align}\label{eqn:stoptime}
 C^{-1}|I|N\leq   \#\{i: \gamma_i(s\wedge \tau)\in I\}, \#\{i: \lambda_i(s\wedge \tau)\in I\} \leq C|I|N.
  \end{align}

  For any configuration $\bm \xi$ with $n$ particles we define
  \begin{align*}
   \quad \varphi_s(\bm \xi)\deq\sum_{\alpha=1}^{n}\varphi_{x_\alpha}(\lambda_{y_\alpha}(s\wedge \tau)), \quad \phi_s(\bm \xi)\deq e^{\nu \varphi_{s}(\bm \xi)},\quad v_s(\bm\xi)\deq\phi_s(\bm \xi)r_{s\wedge\tau}(\bm \eta, \bm \xi),\quad 
   X_s \deq\sum_{\bm \xi}\pi(\bm \xi)v_s(\bm \xi)^2,
  \end{align*}
  where $\pi$ is the reversible measure with respect to the eigenvector moment flow (as in \eqref{eqn:weight}). 
  
 \emph{Second Step:} We denote $X_t^*:=\sup_{t_0\leq s\leq t}X_s$ (by our definition, $X_s$ is always positive). We claim that \eqref{eqn:fsp} follows from the estimate
\begin{align}\label{eqn:upperbound}
\bE[X_t^*]\leq C e^{C(t-t_0)\nu \log N},
\end{align}  
where $C$ is a constant depending on $n$. In fact, \eqref{eqn:upperbound} implies that
  \begin{align}\label{e:marest}
   \bE\left[\sup_{t_0\leq s\leq t}e^{2N\nu\sum_{\alpha=1}^{n}\varphi_{x_\alpha}(\lambda_{y_\alpha}(s\wedge \tau))}r_{s\wedge \tau}^2(\bm \eta, \bm \xi)\right]\leq  Ce^{C(t-t_0)\nu\log N}.
  \end{align}
Under the assumption that $\tilde{d}(\boeta,\boxi)\geq \psi\ell/2$, there exists some index $1\leq \alpha\leq n$ (by symmetry, we can assume $x_\alpha\leq y_\alpha$) such that
  \begin{align*}
  \#\{i: \gamma_i(t_0)\in I^r_{\kappa}(E_0), i\in \qq{x_\alpha, y_\alpha}\}\geq \psi \ell/2,
  \end{align*}
 then it follows from \eqref{eqn:stoptime} that $|[\gamma_{x_\alpha}(t_0),\gamma_{y_\alpha}(t_0)]\cap I_{2\kappa}^r(E_0)|\gtrsim\psi\ell/N$, and thus $ \varphi_{x_\alpha}(\gamma_{y_{\alpha}}(t_0))-\varphi_{x_\alpha}(\gamma_{x_{\alpha}}(t_0))\gtrsim\psi \ell/N$.  We can lower bound $\varphi_{x_\alpha}(\lambda_{y_{\alpha}}(s\wedge \tau))$ as
\begin{align}\begin{split}\label{e:lowbound1}
\varphi_{x_\alpha}(\lambda_{y_{\alpha}}(s\wedge \tau))
\geq& \varphi_{x_\alpha}(\gamma_{y_{\alpha}}(t_0))-|\varphi_{x_\alpha}(\lambda_{y_{\alpha}}(s\wedge \tau))-\varphi_{x_\alpha}(\gamma_{y_{\alpha}}(s\wedge\tau))|\\
-&|\varphi_{x_\alpha}(\gamma_{y_{\alpha}}(s\wedge \tau))-\varphi_{x_\alpha}(\gamma_{y_{\alpha}}(t_0))|.
\end{split}\end{align}
For the second term in \eqref{e:lowbound1}, since either $\gamma_{y_\al}(s\wedge \tau)\in I_{\kappa}^{r}(E_0)$, and $|\lambda_{y_\al}(s\wedge \tau)-\gamma_{y_\al}(s\wedge \tau)|\leq \psi/N$, or $\gamma_{y_\al}(s\wedge \tau)\notin I_{\kappa}^{r}(E_0)$, and $\varphi_{x_\alpha}(\lambda_{y_{\alpha}}(s\wedge \tau))=\varphi_{x_\alpha}(\gamma_{y_{\alpha}}(s\wedge \tau))=0$. In both cases $|\varphi_{x_\alpha}(\lambda_{y_{\alpha}}(s\wedge \tau))-\varphi_{x_\alpha}(\gamma_{y_{\alpha}}(s\wedge\tau))|\lesssim \psi/N$. For the third term in \eqref{e:lowbound1}, we have
\begin{align*}
|\varphi_{x_\alpha}(\gamma_{y_{\alpha}}(s\wedge \tau))-\varphi_{x_\alpha}(\gamma_{y_{\alpha}}(t_0))|
\leq \int_{t_0}^s|\varphi'(\gamma_{y_\al}(\sigma\wedge\tau)||\gamma_{y_\al}'(\sigma\wedge \tau)|(s-t_0)\lesssim \log N \ll \psi\ell/N,
\end{align*}
where we used \eqref{e:dergamma}. As a consequence we have $\varphi_{x_\alpha}(\lambda_{y_{\alpha}}(s\wedge \tau))
\gtrsim \psi\ell/N$, for any $t_0\leq s\leq t$. It then follows by combining with \eqref{e:marest}, 
  \begin{align*}
  \bE[\sup_{t_0\leq s\leq t }r_{s\wedge \tau}(\bm \eta, \bm \xi)^2]\leq  e^{-c\psi}.  \end{align*}
Since $\bm\la(t_0)$ is a good eigenvalue configuration, with overwhelming probability we have $\tau\geq t$, Therefore, \eqref{eqn:fsp} follows by the Markov inequality.


\emph{Third Step:} In the following we prove \eqref{eqn:upperbound}. We decompose $X_s$ as $X_s=M_s+A_s$, where $M_s$ is a continuous local martingale with $M_{t_0}=0$, and $A_s$ is a continuous adapted process of finite variance. We denote $A_t^*:=\sup_{t_0\leq s\leq t}A_s$, and $M_t^*:=\sup_{t_0\leq s\leq t}|M_t|$. Then we have that $X_t^*\leq M_t^*+A_t^*$. For $M_t^*$ we will bound it by Burkholder-Davis-Gundy inequality:
\begin{align}\label{eqn:BDG}
\bE\left[\left(M_t^*\right)^2\right]\leq C\, \bE\left[\int_{t_0}^t\langle \rd M_s, \rd M_s\rangle\right].
\end{align}
For $A_t^*$, since $A_t$ is a finite variance process, we will directly upper bound $\del_s A_s$, and
\begin{align}\label{eqn:upA*}
\bE\left[A_t^*\right] \leq \bE\left[A_{t_0}+\int_{t_0}^t (\del_s A_s \vee 0 )\rd s\right].
\end{align}

By It\'{o}'s formula we have
  \begin{align}
   \rd X_s =
          &\sum_{\bm \xi}\pi (\bm\xi)\sum_{|k-j|\leq \ell}c_{kj}2\xi_k(1+2\xi_j)\left(\frac{\phi_s(\bm \xi^{kj})}{\phi_s(\bm\xi)} +\frac{\phi_s(\bm\xi)}{\phi_s(\bm\xi^{kj})}-2\right) v_s(\bm \xi^{kj})v_s(\bm \xi)\rd (s\wedge \tau)\label{line2}\\
          +&\sum_{\bm\xi}\pi(\bm\xi)r_{s\wedge \tau}^{2}(\boeta,\boxi)\langle \rd \phi_s(\bm \xi), \rd\phi_s(\bm \xi)\rangle\label{line4}\\
          +&2\sum_{\bm\xi} \pi(\bm\xi)v_s(\bm\xi)r_{s\wedge \tau}(\bm \eta,\boxi)\rd \phi_s(\bm \xi)\label{line3}\\
          -&\sum_{\bm \xi}\pi(\bm \xi)\sum_{|k-j|\leq \ell}c_{kj}2\xi_k(1+2\xi_j)(v_s(\bm \xi^{kj})-v_s(\bm \xi))^2\rd (s\wedge \tau).\label{line1}
  \end{align}
The martingale part comes from \eqref{line3},
\begin{align*}
\rd M_s=2\sum_{\bm\xi} \pi(\bm\xi)v_s(\bm\xi)^2 \nu \sum_{\alpha=1}^n \varphi_{x_\alpha}'(\lambda_{y_{\alpha}}(s\wedge r))\frac{\rd b_{y_\alpha y_\alpha}(s\wedge\tau)}{\sqrt{N}}.
\end{align*}
Since $\|\varphi'_i\|_{\infty}\leq 1$, we have
\begin{align*}
\langle \rd M_s, \rd M_s\rangle\lesssim_n \frac{\nu^2}{N}X_s^2\rd s\wedge \tau.
\end{align*}
Therefore, combining with $\eqref{eqn:BDG}$, we have 
\begin{align}\label{eqn:quadX}
\bE\left[\left(M_t^*\right)^2\right]\lesssim_n \frac{\nu^2}{N}\bE\left[\int_{t_0}^t  X_s^2\rd s\right]=\frac{\nu^2}{N}\int_{t_0}^t \bE[X_s^2]\rd s
\end{align}

To understand \eqref{eqn:upA*} and \eqref{eqn:quadX}, we need an upper bound of $\del A_s$, which is the finite variance part of $\rd X_s$. Thanks to the choice of $\varphi_i$'s, we can directly upper bound \eqref{line2} and \eqref{line4} in terms of $X_s$. For \eqref{line3}, we upper bound it by taking advantage of its cancellation with \eqref{line1}.

Firstly, for \eqref{line2}, we need the following estimate: for $|k-j|\leq \ell$,
  \begin{align}\label{e:cutbound}
   \left|\frac{\phi_s(\bm \xi^{kj})}{\phi_s(\bm\xi)} +\frac{\phi_s(\bm\xi)}{\phi_s(\bm\xi^{kj})}-2\right|\lesssim \nu^2|\lambda_k-\lambda_j|^2.
  \end{align}
  We assume that $j<k$, then there exists $1\leq p<q\leq n$ such that $y_{p-1}\leq j<y_{p}$ (we set $y_0=0$) and $y_{q-1}<k=y_q$ (recall $y_q= y_q(\boxi)$) and
  \begin{align*}
   |\varphi_{s}(\bm \xi^{kj})-\varphi_s(\bm\xi)|\leq \sum_{\alpha=p}^{q}|\varphi_{x_\alpha}(\lambda_{y_{\alpha-1}\vee j})-\varphi_{x_\alpha}(\lambda_{y_\alpha})|
  \end{align*}
Since $y_{\alpha}-(y_{\alpha-1}\vee j)\leq k-j\leq \ell$, by our choice of stopping time $\tau$,  if $\lambda_{y_{\alpha-1}\vee j}\leq E_0-(1-\kappa)r$, then $\lambda_{y_{\alpha}}\leq E-(1-\kappa)r+C\ell/N$, where $C$ is from \eqref{eqn:stoptime}, and both $\varphi_{x_\alpha}(\lambda_{y_{\alpha-1}\vee j})$ and $\varphi_{x_\alpha}(\lambda_{y_\alpha})$ vanish. Especially we have  $\varphi_{x_\alpha}(\lambda_{y_{\alpha-1}\vee j})-\varphi_{x_\alpha}(\lambda_{y_\alpha})=0$. Similarly, if $\lambda_{y_{\alpha}}\geq E+(1-\kappa)r$, then $\lambda_{y_{\alpha-1}\vee j}\geq E+(1-\kappa)r-C\ell/N$, and $\varphi_{x_\alpha}(\lambda_{y_{\alpha-1}\vee j})-\varphi_{x_\alpha}(\lambda_{y_\alpha})=0$. Therefore, 
 \begin{align*}
   |\varphi_{s}(\bm \xi^{kj})-\varphi_s(\bm\xi)|\lesssim \left|[\lambda_{y_p\vee j}(s\wedge \tau), \lambda_{y_q}(s\wedge \tau)]\cap I_{
   \kappa}^r(E_0)\right|\lesssim \min \{|\lambda_j(s\wedge \tau)- \lambda_k(s\wedge\tau)|, \nu^{-1}\}.
  \end{align*}
  where we used \eqref{eqn:stoptime} again.  This estimate leads to \eqref{e:cutbound}:
  \begin{align*}
   \left|\frac{\phi_s(\bm \xi^{kj})}{\phi_s(\bm\xi)} +\frac{\phi_s(\bm\xi)}{\phi_s(\bm\xi^{kj})}-2\right|
   =\left|\exp{\frac{\nu(\varphi_{s}(\bm \xi^{kj})-\varphi_s(\bm\xi))}{2}}- \exp{\frac{\nu(\varphi_{s}(\bm \xi)-\varphi_s(\bm\xi^{kj}))}{2}}\right|^2\lesssim \nu^2|\lambda_k-\lambda_j|^2.
  \end{align*}
Combining with  \eqref{e:cutbound}, it follows that 
  \begin{align}\label{eqn:line2bound}
   \eqref{line2}\lesssim \frac{\nu^2}{N}\sum_{\bm \xi}\pi (\bm\xi)\sum_{|k-j|\leq \ell}2\xi_k(1+2\xi_j) v_s(\bm \xi^{kj})v_s(\bm \xi)\rd (s\wedge \tau)\lesssim_n \frac{\nu^2\ell}{N}X_s\rd (s\wedge \tau).
  \end{align}

  For \eqref{line4}, we have the bound
  \begin{align}\label{eqn:line4bound}
   \eqref{line4}=\nu^2 X_s \sum_{\alpha=1}^{n}\frac{\varphi_{x_\alpha}'^2(\lambda_{y_\alpha}(s))}{N}\rd (s\wedge \tau)\lesssim_n \frac{\nu^2}{N}X_s\rd (s\wedge \tau).
  \end{align}

  For \eqref{line3}, the finite variance part is given by
  \begin{align}
   \label{line5}&\sum_{\bm \xi}\pi(\bm\xi)v_s(\bm\xi)^2\frac{1}{N}\sum_{\alpha=1}^n\left(\nu \varphi_{x_\alpha}''(\lambda_{y_{\alpha}})+\nu^2\varphi_{x_\alpha}'^2(\lambda_{y_\alpha})\right)\rd (s\wedge \tau)\\
   \label{line6}+&2\nu\sum_{\bm \xi}\pi(\bm\xi)v_s(\bm\xi)^2\sum_{\alpha=1}^n\varphi_{x_\alpha}'(\lambda_{y_\alpha})\frac{1}{N}\sum_{|k-y_\alpha|>\ell}\frac{\rd (s\wedge \tau)}{\lambda_{y_\alpha}-\lambda_k}\\
   \label{line7}+&2\nu\sum_{\bm \xi}\pi(\bm\xi)v_s(\bm\xi)^2\sum_{\alpha=1}^n\varphi_{x_\alpha}'(\lambda_{y_\alpha})\frac{1}{N}\sum_{0<|k-y_\alpha|\leq \ell}\frac{\rd (s\wedge \tau)}{\lambda_{y_\alpha}-\lambda_k}.
  \end{align}
  By our choice of the cutoff function,  $|\nu \varphi_{x_\alpha}''(\lambda_{y_{\alpha}})+\nu^2\varphi_{x_\alpha}'^2(\lambda_{y_\alpha})|\lesssim \nu^2$
  \begin{align}\label{eqn:line5bound}
   \eqref{line5}\lesssim_n \frac{\nu^2}{N}X_s\rd(s\wedge \tau).
  \end{align}

  For \eqref{line6}, we either have  $\lambda_{y_\alpha}\notin I_{2\kappa}^{r}(E_0)$, then $\varphi_{x_\alpha}'(\lambda_{y_\alpha})=0$; or $\lambda_{y_\alpha}\in I_{2\kappa}^{r}(E_0)$, in this case, by a dyadic decomposition argument similar to \eqref{e:dyadic}, we have :
  \begin{align*}
   \left|\frac{1}{N}\sum_{k:|k-y_\alpha|>\ell}\frac{1}{\lambda_{y_\alpha}(s\wedge \tau)-\lambda_k(s\wedge\tau)}\right|\lesssim \log N.
  \end{align*}
  Therefore we always have that
  \begin{align}\label{eqn:line6bound}
   \eqref{line6}\lesssim \nu\log N X_s\rd (s\wedge \tau).
  \end{align}

  Finally to bound \eqref{line7}, we symmetrize its summands
  \begin{align}
\notag & 2\nu\sum_{\bm \xi}\pi(\bm\xi)v_s(\bm\xi)^2\sum_{\alpha=1}^n\varphi_{x_\alpha}'(\lambda_{y_\alpha})\frac{1}{N}\sum_{0<|k-y_\alpha|\leq \ell}\frac{\rd (s\wedge \tau)}{\lambda_{y_\alpha}-\lambda_k}\\
\notag = &\frac{2\nu}{N}\sum_{0<k-i\leq \ell}\frac{\rd (s\wedge  \tau)}{\la_i-\la_k}
\sum_{\boxi}\pi(\boxi)v_s(\boxi)^2 \sum_{\alpha: y_\alpha=i}\varphi'_{x_\alpha}(\la_i)
+\frac{2\nu}{N}\sum_{0<i-k\leq \ell}\frac{\rd (s\wedge \tau)}{\la_i-\la_k}
\sum_{\boxi}\pi(\boxi)v_s(\boxi)^2 \sum_{\alpha: y_\alpha=k} \varphi'_{x_\alpha}(\la_k) \\
\notag=&  \frac{2\nu}{N}\sum_{0<k-i\leq \ell}\frac{\rd (s\wedge  \tau)}{\la_i-\la_k}
\sum_{\boxi}\pi(\boxi)v_s(\boxi)^2\left(\sum_{\alpha: y_\alpha=i}\varphi'_{x_\alpha}(\la_i)
-
\sum_{\alpha:y_\alpha=k}\varphi'_{x_\alpha}(\la_k)
\right)\\
\leq & \frac{2\nu}{N}\sum_{0<k-i\leq \ell}\frac{\rd (s\wedge  \tau)}{\la_i-\la_k}
\sum_{\boxi}\pi(\boxi)v_s(\boxi)^2\left(\sum_{\alpha: y_\alpha=i}\varphi'_{x_\alpha}(\la_i)
-
\sum_{\alpha:y_\alpha=k}\varphi'_{x_\alpha}(\la_i) \right)
+
O(n\nu X_s\rd (s\wedge  \tau)),\label{eqn:symm}
  \end{align}
where in the last inequality, we replaced $\varphi'_{x_\alpha}(\la_k)$ by $\varphi'_{x_\alpha}(\la_i)$. By our choice of $\varphi_i$, $|\varphi'_{x_\alpha}(\la_i)-\varphi'_{x_\alpha}(\la_k)|\leq \|\varphi''(x_{\alpha})\|_{\infty}|\la_i-\la_k|\leq \nu|\la_i-\la_k|$, and there are at most $2\ell n$ choices for the pairs $(k,i)$, so the error is at most 
$O(\nu^2\ell X_s/N)=O(n\nu X_s)$.

In all the following bounds, we
consider $i$ and 
$k$ as fixed indices. We also introduce the following subsets of configurations with $n$ particles, for any $0\leq q\leq p\leq n$:
$$
\mathcal{A}_{p}=\{\boxi:\xi_i+\xi_k=p\},\
\mathcal{A}_{p,q}=\{\boxi\in\mathcal{A}_p:\xi_i=q\}.
$$
We denote $\bar\boxi=(\bar\xi_1,\bar\xi_2,\cdots, \bar\xi_N)$  the configuration exchanging all particles from sites $i$ and $k$, i.e. $\bar\xi_i=\xi_k$, $\bar\xi_k=\xi_i$ and $\bar\xi_j=\xi_j$ if $j\neq i, k$. We denote the locations of particles of the configuration $\bar{\bm\xi}$: $1\leq \bar {y}_1\leq \bar{y}_2\cdots \leq \bar{y}_n\leq N$.
Using  $\pi(\boxi)=\pi(\bar\boxi)$, we can rewrite  the sum over $\boxi$ in (\ref{eqn:symm}) as
\begin{align}
\notag\frac{2}{\la_i-\la_k}\sum_{p=0}^n\sum_{q=0}^p\sum_{\boxi\in\mathcal{A}_{p,q}}\pi(\boxi)v_s(\boxi)^2
\left( \sum_{\alpha: y_\alpha=i}\varphi'_{x_\alpha}(\la_i)
-
\sum_{\alpha:y_\alpha=k}\varphi'_{x_\alpha}(\la_i) \right)\\
\notag=
\frac{1}{\la_i-\la_k}\sum_{p=0}^n\sum_{q=0}^{p}
\sum_{\boxi\in\mathcal{A}_{p,q}}\pi(\boxi)  
 v_s(\boxi)^2  
\left(\sum_{\alpha: y_\alpha=i}\varphi'_{x_\alpha}(\la_i)  
-
\sum_{\alpha:y_\alpha=k}\varphi'_{x_\alpha}(\la_i) \right)  \\
-
\frac{1}{\la_i-\la_k}\sum_{p=0}^n\sum_{q=0}^{p}
\sum_{\boxi\in\mathcal{A}_{p,q}}\pi(\boxi)  
 v_s(\bar{\bm \xi})^2  
\left(\sum_{\alpha: \bar{y}_\alpha=k}\varphi'_{x_\alpha}(\la_i)  
-
\sum_{\alpha:\bar{y}_\alpha=i}\varphi'_{x_\alpha}(\la_i) \right). \label{eqn:interm1}
\end{align}
For $i<k$, both index sets $\{\alpha: y_\alpha=k\}\cup\{\alpha:\bar{y}_\alpha=k\}$ and $\{\alpha: y_\alpha=i\}\cup\{\alpha:\bar{y}_\alpha=i\}$ has cardinality $p$, and the $j$-th largest number in the first set is larger than its counterpart in the second set. By \eqref{monopsi},  for any $a\leq b$, we have $\varphi'_{a}(x)\geq \varphi'_{b}(x)$.  This implies that
\begin{equation}\label{eqn:interm2}
\sum_{\alpha: y_\alpha=i}\varphi'_{x_\alpha}(\la_i)
-
\sum_{\alpha:y_\alpha=k}\varphi'_{x_\alpha}(\la_i) 
\geq
\sum_{\alpha:  \bar y_\alpha=k}\varphi'_{x_\alpha}(\la_i)
-
\sum_{\alpha:  \bar y_\alpha=i}\varphi'_{x_\alpha}(\la_i).
\end{equation}
Equations (\ref{eqn:interm1}) and (\ref{eqn:interm2}) together with $\la_i<\la_k$ give
\begin{align}
\notag&\quad\frac{2}{\la_i-\la_k}\sum_{p=0}^n\sum_{q=0}^p\sum_{\boxi\in\mathcal{A}_{p,q}}\pi(\boxi)v_s(\boxi)^2
\left( \sum_{\alpha: y_\alpha=i}\varphi'_{x_\alpha}(\la_i)
-
\sum_{\alpha:y_\alpha=k}\varphi'_{x_\alpha}(\la_i) \right)
\\
\notag&\leq
\frac{1}{\la_i-\la_k}\sum_{p=0}^n\sum_{q=0}^{p}\sum_{\boxi\in\mathcal{A}_{p,q}}\pi(\boxi)\left(
v_s(\boxi)^2-v_s(\bar\boxi)^2 \right)
\left( \sum_{\alpha: y_\alpha=i}\varphi'_{x_\alpha}(\la_i)
-
\sum_{\alpha:y_\alpha=k}\varphi'_{x_\alpha}(\la_i) \right) \\
&\lesssim_n
\frac{1 }{|\la_i-\la_k|}\sum_{\boxi}  \pi(\boxi) \left|
v_s(\boxi)^2-v_s(\bar\boxi)^2 \right|,\label{line8}
\end{align}
where we used, in the second inequality, $\|\psi_{x_\alpha}'\|_\infty\leq 1$.

Note that transforming $\boxi$ into $\bar\boxi$ can be achieved by transferring a particle for $i$ to $k$ (or $k$ to $i$) one by one at most $n$ times. More precisely, if $\boxi\in \cal A_{p,q}$ such that $q\leq p-q$, we can define $\bm \xi_{j+1}=\bm \xi_{j}^{ki}$, for $0\leq j\leq p-2q$. Then $\bm \xi_0=\bm \xi$ and $\bm \xi_{p-2q}=\bar{\bm \xi}$ and
\begin{align*}\pi(\boxi)  |
v_s(\boxi)^2-v_s(\bar\boxi)^2|\leq 
\sum_{j=0}^{p-2q-1}\pi(\boxi)  |
v_s(\boxi_j)^2-v_s(\boxi_j^{ki})^2|\lesssim_n \sum_{j=0}^{p-2q-1}\pi(\boxi_j)  |
v_s(\boxi_j)^2-v_s(\boxi_j^{ki})^2|
\end{align*}
where in the last inequality we used $\pi(\bm \xi)\lesssim_n \pi(\bm \xi_j)$. Therefore we can bound  \eqref{line8} as
\begin{align*}
&\frac{1}{|\la_i-\la_k|}\sum_{\boxi}  \pi(\boxi)  \left|
v_s(\boxi)^2-v_s(\bar\boxi)^2 \right|\leq
\frac{C}{|\la_i-\la_k|} \sum_{\boxi}  \pi(\boxi) |v_s(\boxi)^2-v_s(\boxi^{ki})^2|\\
\leq&
\ell\sum_{\boxi}   \pi(\boxi)  \frac{\left(v_s(\boxi)-v_s(\boxi^{ki})\right)^2}{N(\la_i-\la_k)^2}+
\frac{C^2N}{\ell}\sum_{\boxi}  \pi(\boxi) \left(v_s(\boxi)+v_s(\boxi^{ki})\right)^2
\end{align*}
where we used AM-GM inequality.  Finally, we obtain the following bound
  \begin{align}\label{eqn:line7bound}
   \eqref{line7}\leq 2\sum_{\bm \xi}\pi(\bm \xi)\sum_{0<|i-k|\leq \ell}\frac{(v_s(\bm \xi)-v_s(\bm \xi^{ki}))^2}{N(\lambda_i-\lambda_k)^2}\rd (s\wedge \tau)+C\nu X_s\rd (s\wedge \tau).
  \end{align}
Notice that the first term in \eqref{eqn:line7bound} combined with \eqref{line1} give a negative contribution. Thus, \eqref{eqn:line5bound}, \eqref{eqn:line6bound} and \eqref{eqn:line7bound} together lead to 
\begin{align}\label{eqn:line3bound}
\eqref{line3}+\eqref{line1}\lesssim_n \nu \log N X_s\rd (s\wedge \tau)
\end{align}
  
  \eqref{eqn:line2bound}, \eqref{eqn:line4bound} and \eqref{eqn:line3bound}, all together, give the following upper bound  on the finite variance part of $X_s$:
  \begin{align}\label{eqn:dAbound}
   \del_s A_s\lesssim_n \nu \log N X_s.
  \end{align}
  
  With \eqref{eqn:dAbound}, it is easy to estimate $\bE[X_t]$ and $\bE[X_t^2]$. For $X_t$, by taking expectation on both sides of \eqref{eqn:dAbound}, we have
 \begin{align*}
   \del_s \bE[X_s]\lesssim_n \nu \log N \bE[X_s].
  \end{align*}
  Therefore 
  \begin{align}\label{eqn:X}
  \bE[X_t]\lesssim_n e^{C(t-t_0)\nu\log N},
  \end{align}
   following from Gronwall's inequality, where $C$ is a constant depending on $n$. And we used that by rigidity we have $|\la_i(t_0)-\gamma_{i}(t_0)|\leq \psi/N$, and thus $|\varphi_i(\la_i(t_0))|\leq |\varphi_i(\gamma_i(t_0))|+|\varphi_i(\la_i(t_0))-\varphi_i(\gamma_i(t_0))|\lesssim 1/\nu$, and $\bE[X_{t_0}]\lesssim_n1$.
Similarly for $X_t^2$, by  It\'{o}'s formula, we have
\begin{align}
\notag \rd X_s^2
=& 2X_s\rd X_s +\langle \rd X_s, \rd X_s\rangle\\
\label{eqn:square}=&2X_s \rd M_s + 2X_s\rd A_s +\langle \rd M_s, \rd M_s \rangle.
\end{align}
We take expectation on both sides of \eqref{eqn:square} and obtain,
\begin{align*}
\del_s \bE[X_s^2] \lesssim_n \nu \log N \bE[X_s^2] +\nu^2N^{-1} \bE[X_s^2]\lesssim_n \nu \log N \bE[X_s^2] .
\end{align*}
Again by Gronwall's inequality, we have 
\begin{align}\label{eqn:X2}
\bE[X_t^2]\lesssim_n e^{C(t-t_0)\nu\log N}
\end{align}
\eqref{eqn:BDG}, \eqref{eqn:quadX}, \eqref{eqn:X} and \eqref{eqn:X2} together implies: 
\begin{align*}
\bE[X_t^*]\leq &\bE\left[\left(M_t^*\right)^2\right]^{1/2}+\bE[A_t^*]\\
\leq& \frac{\nu}{\sqrt{N}}\left(\int_{t_0}^t\bE[X_s^2]\rd s\right)^{1/2}+\bE[X_{t_0}]+\nu \log N \int_{t_0}^t\bE[X_s]\rd s 
\lesssim_n e^{C(t-t_0)\nu \log N}
\end{align*}
This finishes the proof of \eqref{eqn:upperbound}.
\end{proof}

We can take the event $A_2$ of trajectories $(\bm\la(s))_{0\leq s\leq t}$ such that: conditioning on the trajectories $(\bm \la(s))_{0\leq s\leq t}$, the short-range operator $\rU_{\mathscr{S}}$ satisfies,
\begin{align}\label{weakfsp}
\sup_{t_0\leq s\leq t}\rU_{\mathscr{S}}(t_0,s)\delta_{\boeta}(\boxi)\leq e^{-2c\psi}
\end{align}
for any pair of $n$ particle configurations $\boeta$ and $\boxi$ (notice that the total number of $n$ particle configurations is bounded by $N^n$) such that $\tilde{d}(\boeta,\boxi)\geq \psi \ell/2$.
Since with overwhelming probability $\bm\la(t_0)$ is a good eigenvalue configuration, combining with the Lemma \ref{l:fspeed}, we know that $A_2$ holds with overwhelming probability.

Thanks to the semi-group property of $\rU_{\mathscr{S}}$, for any $(\bm\la(s))_{0\leq s\leq t}\in A_2$, we claim, for $N$ large enough, the following hold: conditioning on the trajectories $(\bm \la(s))_{0\leq s\leq t}$, the short-range operator $\rU_{\mathscr{S}}$ satisfies, 
\begin{align}\label{modfsp}
\sup_{t_0\leq s'\leq s\leq t}\rU_{\mathscr{S}}(s',s)\delta_{\boeta}(\boxi)\leq e^{-c\psi},
\end{align}
for any pair of $n$ particle configurations $\boeta$ and $\boxi$ such that $\tilde{d}(\boeta,\boxi)\geq \psi \ell$. We prove the statement by contradiction. Assume there is a pair $\boeta_0$ and $\boxi_0$ with $\tilde{d}(\boeta_0,\boxi_0)\geq \psi \ell$ and time $t_0\leq s'\leq s\leq t$ such that \eqref{modfsp} fails. We take a function 
\begin{align*}
h=\sum_{\tilde{d}(\boeta,\boeta_0)\leq \psi\ell/2}\delta_{\boeta}
\end{align*}
 on the space of $n$ particle configurations. By triangular inequality, for any $\boeta$ such that $\tilde{d}(\boeta,\boeta_0)\leq \psi\ell/2$, we have $\tilde{d}(\boeta,\boxi_0)\geq \psi\ell/2$. Therefore by \eqref{weakfsp},  for sufficiently large $N$,
 \begin{align}\label{e:smallfun}
 \rU_{\mathscr{S}}(t_0,s)h(\boxi_0)\leq N^n e^{-2c\psi}.
 \end{align}
By the same argument for \eqref{e:smallfun}, we have
\begin{align*}
\rU_{\mathscr{S}}(t_0,s')\left(\sum_{\tilde{d}(\boeta,\boeta_0)>\psi \ell/2}\delta_{\boeta}\right)(\boeta_0)
\leq N^n e^{-2c\psi}\leq \frac{1}{2}.
\end{align*}
Notice that $\rU_{\mathscr{S}}(t_0,s')$ preserves the constant function, we have
\begin{align*}
\rU_{\mathscr{S}}(t_0,s)h(\boxi_0)
=\rU_{\mathscr{S}}(s',s)\rU_{\mathscr{S}}(t_0,s')\left(\mathbf1-\sum_{\tilde{d}(\boeta,\boeta_0)>\psi \ell/2}\delta_{\boeta}\right)(\boxi_0)
\geq \frac{1}{2}\rU_{\mathscr{S}}(s',s)\delta_{\boeta_0} (\boxi_0)\geq e^{-c\psi}/2,
\end{align*}
which gives a contradiction with \eqref{e:smallfun}. Therefore, we have the following corollary of Lemma \ref{l:fspeed}:
\begin{corollary}
For any trajectory $(\bm\la(s))_{0\leq s\leq t}\in A_2$ as defined in \eqref{weakfsp}, conditioning on $(\bm\la(s))_{0\leq s\leq t}$, the short-range operator $\rU_{\mathscr{S}}$ satisfies: uniformly, for any function $h$ on the space of $n$ particle configurations, and particle configuration $\boxi$ which is away from the support of $h$ in the sense that $\tilde{d}(\boeta,\boxi)\geq \psi \ell$, for any $\boeta$ in the support of $h$, it holds
\begin{align*}
\sup_{t_0\leq s'\leq s\leq t}\rU_{\mathscr{S}}(s',s)h(\boxi)\leq \|h\|_\infty N^n e^{-c\psi}.
\end{align*}
\end{corollary}

\subsection{Short time relaxation}

\begin{lemma}\label{l:regpath}
Under the Assumption \ref{a:boundImm}, for any $\eta_*\ll t\ll r$,  we fix time $t_0$ and the range parameter $\ell$, such that $\eta_*\ll t_0\leq t\leq t_0+\ell/N\ll r$. 
The Dyson Brownian motion $W_s$ (as in \eqref{DBM}) for $0\leq s\leq t$ induces a measure on the space of eigenvalues and eigenvectors $(\bm \lambda(s), \bm u(s))$ for $0\leq s\leq t$. The following event $A$ of trajectories holds with overwhelming probability:
\begin{enumerate}
\label{rigid}\rm\item The eigenvalue rigidity estimate holds:  $\sup_{t_0\leq s\leq t}|m_s(z)-m_{{\rm{fc}},s}(z)|\leq \psi (N\eta)^{-1}$  uniformly for $z\in \cal D_\kappa$; and $ \sup_{t_0\leq s\leq t}|\lambda_{i}(s)-\gamma_{i}(s)|\leq\psi N^{-1}$ uniformly for indices $i$ such that $\gamma_i(s)\in I_r^{\kappa}(E_0)$.
\rm\item When we condition on the trajectory $\bm \lambda\in A$, with overwhelming probability, the following holds 
\begin{align}\label{evcontrol}
\sup_{t_0\leq s\leq t}|\langle \bmq, G(s,z)\bmq\rangle-m_{{\rm{fc}},s}(z)|\leq \frac{1}{N^{\frak{b}}}+\frac{\psi^2}{\sqrt{N\eta}}
\end{align}
uniformly for $z\in \cal D_\kappa$.
\rm\item Finite speed of propagation holds: uniformly, for any function $h$ on the space of $n$ particle configurations, and particle configuration $\boxi$ which is away from the support of $h$ in the sense that $\tilde{d}(\boeta,\boxi)\geq \psi \ell$, for any $\boeta$ in the support of $h$, it holds
\begin{align}\label{finitespeed}
\sup_{t_0\leq s'\leq s\leq t}\rU_{\mathscr{S}}(s',s)h(\boxi)\leq \|h\|_\infty N^n e^{-c\psi}.
\end{align}
\end{enumerate}
\end{lemma}
Let the indices $b_1$, $b_2$, $b_1-d_1$ and $b_2+d_2$ be such that among all the classical eigenvalue locations at time $t_0$,  $\gamma_{b_1}(t_0)$, $\gamma_{b_2}(t_0)$, $\gamma_{b_1-d_1}(t_0)$ and $\gamma_{b_2+d_2}(t_0)$ are closest to $E_0-(1-7\kappa/4)r$, $E_0+(1-7\kappa/4)r$, $E_0-(1-3\kappa/2)r$ and $E_0+(1-3\kappa/2)r$ respectively. We further define $d=\min\{d_1,d_2\}$. We collects some facts here, which will be used throughout the rest of this section:
\begin{enumerate}
\item $m_s(z)$ is the Stieltjes transform of empirical eigenvalue distribution $N^{-1}\sum_{i=1}^N\delta_{\lambda_i(s)}$, and $\langle\bmq, G(s,z) \bmq\rangle$ can be viewed as the Stieltjes transform of the weighted spectral measure: 
$\sum_{i=1}^{N}\langle \bmu_i(s),\bmq\rangle^2 \delta_{\lambda_i(s)}$.
The imaginary part of Stieltjes transform contains full information of the spectrum. \eqref{e:dbound} and Lemma \ref{l:regpath} implies the following statements in terms of averaged density of eigenvalues and eigenvectors of $H_s$:
there exists some universal constant $C$ such that for any $t_0\leq s\leq t$, and interval $I$ centered in $I_{\kappa}^r(E_0)$, with $|I|\geq \psi^4/N$, we have
  \begin{align}\label{e:avdeleig}
 C^{-1}|I|N\leq   \#\{i: \gamma_i(s)\in I\}, \#\{i: \lambda_i(s)\in I\} \leq C|I|N.
  \end{align}
especially, $C^{-1}rN\leq d\leq CrN$; and with overwhelming probability
\begin{align}\label{e:avdelev}
C^{-1}|I|\leq \sum_{i:\lambda_i(s)\in I}\langle \bmq, \bmu_i(s)\rangle^2\leq C|I|.
\end{align}
\item Since $\ell/N\ll r$, for any index $i\in \qq{b_1-d-3\psi\ell, b_2+d+3\psi\ell}$, we have $\gamma_i(t_0)\in I_{5\kappa/4}^r(E_0)$. Therefore, for any $t_0\leq s\leq t$, $|\lambda_i(s)-\gamma_i(t_0)|\leq |\lambda_i(s)-\gamma_i(s)|
+|\gamma_i(s)-\gamma_i(t_0)|\leq \psi/N+C\log N(s-t_0)\ll r$, and $\lambda_i(s)\in I_\kappa^r(E_0)$, where we used \eqref{e:dergamma}. Moreover, the eigenvector $\bmu_i(s)$ is localized in the direction $\bmq$ with high probability,
\begin{align}\label{e:deleig}
N\langle q, \bmu_i(s)\rangle^2\leq \psi^4 \Im [\langle \bmq, G(s,(\lambda_i(s)+\i \psi^4/N))\bmq\rangle]\lesssim \psi^4.
\end{align}
\end{enumerate}
%

We define the following flattening and averaging operators on the space of functions of configurations with $n$ points: 
\begin{align}
&({\rm Flat}_a (f))(\boeta)=\left\{
\begin{array}{cc}
f(\boeta), & \ {\rm if }\ \boeta\subset\llbracket b_1-a,   b_2+a\rrbracket,\\
1, &\ {\rm otherwise},
\end{array}
\right.
&{\rm Av}(f)=\frac{1}{d}\sum_{   a\in\llbracket  1 ,  d \rrbracket}{\rm Flat}_a(f).
\end{align}
We can write 
\begin{equation}\label{eqn:ak}
{\rm Av}(f)(\boeta)=a_{\boeta} f(\boeta)+(1-a_{\boeta})
\end{equation}
for some coefficient $a_{\boeta}\in[0,1]$ ($a_{\boeta}=0$ if $\boeta\not\subset \llbracket b_1-d,b_2+d \rrbracket  $, $\; a_{\boeta} =1$ if $
\boeta\subset\llbracket  b_1,  b_2\rrbracket $). We will only use the elementary property
\begin{equation}\label{eqn:propa}
|a_{\boeta}-a_{{\boxi}}|\lesssim d(\boeta,\boxi)/d,
\end{equation}
where the distance is defined in \eqref{dis}.

For a general number of particles $n$, consider now the following modification of the eigenvector moment flow (\ref{ve}).
We only keep the short-range dynamics (depending on the short range parameter $\ell$) and modify the initial condition to be flat when there is a particle outside the interval we are interested, i.e. $[E_0-r, E_0+r]$:
\begin{align}\begin{split}\label{eqn:modMom}
&\partial_t g_{t} =  \mathscr{S}(t) g_{t},\\
&g_{t_0}(\boeta)=({\rm Av}f_{t_0})(\boeta),\end{split}
\end{align}
for $n=1$,  we write these functions as  $f_t(k)$ and $g_t(k)$ when $\boeta$ is the configuration with $1$ particle at $k$.  We remind the reader that $f_t(\boeta)$ can be  define either by  \eqref{feq} or by the solution of 
the equation \eqref{ve}. 

Before we prove our main results, we still need the following lemma on the $L^{\infty}$ control on the difference of the full operator ${\rm U}_{\mathscr B}$ and the short-range operator ${\rm U}_{\mathscr L}$:

\begin{lemma}\label{l:Lbound}
For any eigenvalue trajectory $\bm \lambda\in A$ as defined in Lemma \ref{l:regpath}, we define the eigenvector moment flow as in Theorem \ref{t:emf}. we have the following $L^{\infty}$ control on the difference of the full operator ${\rm U}_{\mathscr B}$ and the short-range operator ${\rm U}_{\mathscr L}$:
\begin{align}\label{Lbound}
\left|\left(\rU_{\mathscr{B}}(t_0,t)f_{t_0}-\rU_{\mathscr{S}}(t_0,t)f_{t_0}\right)(\boxi)\right|\lesssim_n \psi^{4n}N(t-t_0)/\ell
\end{align}
where $\boxi$ is any $n$-particle configuration supported on $\qq{b_1-d-2\psi\ell,b_2+d+2\psi\ell}$.
\end{lemma}
\begin{proof}
By Duhamel's principle
\begin{align*}
\left|\left(\rU_{\mathscr{S}}(t_0,t)f_{t_0}-\rU_{\mathscr{B}}(t_0,t)f_{t_0}\right)(\boxi)\right|
=&\left|\int_{t_0}^t\rU_{\mathscr{S}}(s,t)\mathscr{L}(s)\rU_{\mathscr{B}}(t_0,s)f_{t_0}ds (\boxi)\right|\\
=&\left|\int_{t_0}^t\rU_{\mathscr{S}}(s,t)\mathscr{L}(s)f_{s}ds (\boxi)\right|
\end{align*}

For any $\boeta$ corresponds to the configuration $ \{(i_1,  j_1),  \dots, (i_m , j_m ) \}$, with support in $\qq{b_1-d+3\psi\ell, b_2+d+3\psi\ell}$, i.e., $i_1, i_2,\cdots, i_m\in\qq{b_1-d+3\psi\ell, b_2+d+3\psi\ell}$ . Then by \eqref{e:deleig}, with overwhelming probability, we have $\left(N\langle\bm q, \bmu_{i_p}(s) \rangle^{2}\right)^{j_p} \lesssim \psi^{2j_p}$ uniformly for any $1\leq p\leq m$, which leads to the following priori bound on the eigenvector moment flow:
\begin{align}\label{e:evestimate}
f_s(\boeta)\lesssim \psi^{4n}, \quad f_s(\boeta^{jk})\lesssim \psi^{4n-4}f_t(k).
\end{align}
We remark that $k\in\qq{1,N}$ can be any index. Since \eqref{e:deleig} is local, only holds for eigenvectors corresponding to  eigenvalues in the interval $I_\kappa^r(E_0)$, in general we do not have control on $N\langle \bmq, \bmu_k(s)\rangle^2$. However, it still follows from \eqref{e:evestimate},
\begin{align}
\notag\mathscr{L}(s)f_{s}(\boeta)\lesssim& \sum_{|j-k|\geq \ell}\frac{|f_{s}(\boeta^{jk})-f_{s}(\boeta)|}{N(\lambda_j-\lambda_k)^2}
\leq \sum_{p=1}^m\sum_{k:|i_p-k|\geq \ell}\frac{\psi^{4n-4} f_{s}(k)+\psi^{4n}}{N(\lambda_{i_p}-\lambda_k)^2}.
\end{align}
Notice that $i_p \in\qq{b_1-d+3\psi\ell, b_2+d+3\psi \ell}$, and thus $\lambda_{i_p}(s)\in I_\kappa^r(E_0)$. A similar dyadic decomposition as in \eqref{e:dyadic}, combining with \eqref{e:avdeleig} and \eqref{e:avdelev},  we have
\begin{align*}
\sum_{k:|i_p-k|\geq \ell}\frac{ f_{s}(k)}{N(\lambda_{i_p}-\lambda_k)^2}
&=\sum_{q=1}^{\lceil \log_2 N/\ell\rceil}\sum_{k: 2^{q-1}\ell\leq |k-i_p|\leq 2^q\ell}\frac{ f_{s}(k)}{N(\lambda_{i_p}-\lambda_k)^2}\\
&\lesssim\sum_{q=1}^{\lceil \log_2 N/\ell\rceil}\frac{N}{2^{2q}\ell^2}\sum_{k: 2^{q-1}\ell\leq |k-i_p|\leq 2^q\ell} f_{s}(k)\lesssim \frac{N}{\ell}.
\end{align*}
Similarly, we also have 
\begin{align*}
\sum_{k:|i_p-k|\geq \ell}\frac{ 1}{N(\lambda_{i_p}-\lambda_k)^2}\lesssim \frac{N}{\ell},
\end{align*}
and it follows
\begin{align}\label{eqn:cuttail}
\mathscr{L}(s)f_{s}(\boeta)\lesssim_m \frac{\psi^{4n}N}{\ell}
\end{align}
Notice that $\tilde d (\supp(\mathscr{L}(s)f_{s}-{\rm Flat}_{d+3\psi\ell}(\mathscr{L}(s)f_{s})), \boxi\}\geq \psi\ell$. Therefore by the finite speed of propagation \eqref{finitespeed} in Lemma \ref{l:regpath} of $\rU_{\mathscr{S}}$,  
we have
\begin{align*}
(\rU_{\mathscr{S}}(s,t)\mathscr{L}(s)f_{s}) (\boxi)
=\rU_{\mathscr{S}}(s,t){\rm Flat}_{d+3\psi\ell}(\mathscr{L}(s)f_{s}) (\boxi)+O(e^{-c\psi/2})\lesssim_n \psi^{4n}N/\ell.
\end{align*}
where in the last inequality, we used that $\rU_{\mathscr{S}}$ is a contraction in $L^{\infty}$. \eqref{Lbound} follows, since we gain a factor $t-t_0$ from integration of time.
\end{proof}


By Lemma \ref{l:regpath}, the event $A$ holds with overwhelming probability. Theorem \ref{t:normal} easily follows from the following Theorem.
\begin{theorem}\label{thm:maxPrincipleLoc}
Fix any $\eta_*\ll t\ll r$. For any eigenvalue trajectory $(\bm\la(s))_{0\leq s\leq t}\in A$ defined in Lemma \ref{l:regpath}, let $f$ be a solution of the $\tilde{n}$ particle eigenvector moment flow (\ref{ve}) with initial matrix $H_0$ and eigenvalue trajectories $(\bm\la(s))_{0\leq s\leq t}$.
Then for $N$ large enough we have
\beq\label{fest2}
\sup_{\boeta:\mathcal{N}(\boeta)=\tilde{n},\atop \boeta\subset\llbracket b_1+\psi\ell,b_2-\psi\ell\rrbracket}\left|f_t(\boeta)-1\right|\lesssim_{\tilde{n}}\frac{1}{N^{\frak{d}}},
\bEq
where the constant $\frak{d}>0$ depending on $\fa, \fb, r,t$. 
\end{theorem}

\begin{proof} 
The proof is an induction on the number $\tilde n$ of particles.  
For any $\boeta$ such that $\mathcal{N}(\boeta)=n$ and $\boeta\subset\qq{ b_1+\psi\ell,b_2-\psi\ell}$, we have
\begin{align}
\notag|f_t(\boeta)-g_t(\boeta)|&\leq \left|\left(\rU_{\mathscr{B}}(t_0,t)f_{t_0}-\rU_{\mathscr{S}}(t_0,t)f_{t_0}\right)(\boeta)\right|
+\left|\rU_{\mathscr{S}}(t_0,t)(f_{t_0}-{\rm Av}f_{t_0})(\boeta)\right|\\
\label{eqn:diffdynamic}&\lesssim_n \psi^{4n}N(t-t_0)/\ell+e^{-c\psi/2},
\end{align}
where we bounded the first term by Lemma \ref{l:Lbound}, and the second term by finite speed of propagation \eqref{finitespeed}, since $f_{t_0}-{\rm{Av}}f_{t_0}$ vanishes for any $\boxi$ such that $\boxi\subset \qq{b_1, b_2}$. 

In the following we prove that $\sup_{\boeta}|\{g_t(\boeta)\}-1|\leq N^{-1}$ by a maximum principle argument.
For a given $t_0\leq s\leq t$, let $\tilde{\bm \eta}$, corresponding to the particle configuration $ \{(j_1,  k_1),  \dots, (j_m , k_m )\}$, be such that 
\begin{align*}
g_s(\tilde{ \bm\eta})=\sup_{\bm \eta: {\cal N}(\boeta)=n} \{g_s(\bm \eta)\}.
\end{align*} 
If $g_s(\tilde{\bm\eta})-1\leq N^{-1}$, there is nothing to prove. Otherwise, by finite speed propagation \eqref{finitespeed} in Lemma \ref{l:regpath}, the support of $\tilde{\bm\eta}$ belongs to the interval $\qq{b_1-d-\psi\ell, b_2+d+\psi\ell}$. By the defining relation \eqref{eqn:modMom},
\begin{align}
\notag&\quad \del_s \left(g_s(\tilde{\bm \eta})-1\right) =\mathscr{S}(s)g_s(\tilde{\bm\eta})=\sum_{0<|j-k|\leq \ell} c_{jk} 2\tilde \eta_j (1+2\tilde\eta_k) \left(g_s(\tilde\boeta^{jk})-g_s(\tilde\boeta)\right)\\
\notag\lesssim&\sum_{1\leq p\leq m,\atop k:0<|j_p-k|\leq \ell} \frac{g_s(\tilde \boeta^{j_pk})-g_s(\tilde\boeta)}{N(\lambda_{j_p}-\lambda_k)^2}
\leq \frac{1}{N}\sum_{1\leq p\leq m,\atop k:0<|j_p-k|\leq \ell} \frac{g_s(\tilde \boeta^{j_pk})-g_s(\tilde{\boeta})}{(\lambda_{j_p}-\lambda_k)^2+\eta^2}\\
=&-\frac{1}{N\eta}\left(g_s(\tilde{\bm \eta})-1\right)\sum_{1\leq p\leq m,\atop k:0<|{j_p}-k|\leq \ell}\Im \frac{1}{z_{j_p}-\lambda_k}
+\frac{1}{N\eta}\sum_{1\leq p\leq m,\atop k:0<|{j_p}-k|\leq \ell}\Im \frac{g_s(\tilde{\bm \eta}^{{j_p}k})}{z_{j_p}-\lambda_k}-\Im \frac{1}{z_{j_p}-\lambda_k}\label{eqn:tobound}
\end{align}
where we define  $z_{j_p}=\lambda_{j_p}+\i \eta$, and $\psi^4/N\leq \eta\leq\ell/N$, will be chosen later. For the first term in \eqref{eqn:tobound}
\begin{align*}
\sum_{1\leq p\leq m,\atop k:0<|{j_p}-k|\leq \ell}\Im \frac{1}{z_{j_p}-\lambda_k}
\geq \sum_{p=1}^m\sum_{k:0<|{j_p}-k|\leq \ell}\frac{\eta}{(\lambda_{j_p}-\lambda_k)^2+\eta^2}
\geq \sum_{p=1}^{m}\sum_{k: |\lambda_k-\lambda_{j_p}|\leq\eta} \frac{\eta}{2\eta^2}
\gtrsim N,
\end{align*}
where we used \eqref{e:avdeleig}.
For the second term in \eqref{eqn:tobound}, we claim that for any fixed ${j_p}$ such that $j_p\in\qq{b_1-d-\psi\ell, b_2+d+\psi\ell}$, 
\begin{align}
\begin{split}\label{eqn:keybound}
&\frac{1}{N}\sum_{k:0<|{j_p}-k|\leq \ell}\Im \frac{g_s(\tilde\boeta^{{j_p}k})}{z_{j_p}-\lambda_k}-\Im \frac{1}{z_{j_p}-\lambda_k}\\
=&
 a_{\tilde\boeta}\Im[m_{{\rm fc},s}(z_{j_p})]\left(f_s(\tilde{\boeta}\backslash {j_p})-1\right)+O_{n}\left(\psi^{4n}\left(\frac{1}{N^{\frak{b}}}+\frac{\psi\ell}{d}+\frac{1}{\sqrt{N\eta}}+\frac{N\eta}{\ell}+\frac{N(s-t_0)}{\ell}\right) \right).
\end{split}
\end{align}
We can bound the lefthand side of \eqref{eqn:keybound} 
by $\eqref{eqn:term1}+\eqref{eqn:term2}+\eqref{eqn:term3}$ where
\begin{align}
\label{eqn:term1}&\Im\sum_{k:0<|k-{j_p}|\leq \ell}\frac{1}{N}\frac{(\rU_{\mathscr{S}}(t_0,s){\rm Av}f_{t_0})(\tilde\boeta^{{j_p}k})-({\rm Av}\rU_{\mathscr{S}}(t_0,s)f_{t_0})(\tilde\boeta^{{j_p}k})}{z_{j_p}-\la_k},\\
\label{eqn:term2}&\Im\sum_{k:0<|{j_p}-k|\leq \ell}\frac{1}{N}\frac{({\rm Av}\rU_{\mathscr{S}}(t_0,s)f_{t_0})(\tilde\boeta^{{j_p}k})-({\rm Av}\rU_{\mathscr{B}}(t_0,s)f_{t_0})(\tilde\boeta^{{j_p}k})}{z_{j_p}-\la_k},\\
\label{eqn:term3}&\Im\sum_{k:0<|{j_p}-k|\leq \ell}\frac{1}{N}\frac{({\rm Av}\rU_{\mathscr{B}}(t_0,s)f_{t_0})(\tilde\boeta^{{j_p}k})}{z_{j_p}-\la_k}-\frac{1}{z_{j_p}-\lambda_k}.
\end{align}
The term \eqref{eqn:term1} will be controlled by finite speed of propagation;  \eqref{eqn:term2} will be controlled by  Lemma \ref{l:Lbound}, and \eqref{eqn:term3} by  the isotropic local semicircle law for $\cal N(\boeta)=1$, and by induction for $\cal N(\boeta)\geq 2$. 

To bound \eqref{eqn:term1}, we write   
\begin{align}\begin{split}\label{eqn:(i)}
&(\rU_{\mathscr{S}}(t_0,s){\rm Av}f_{t_0})(\tilde\boeta^{{j_p}k})-({\rm Av}\rU_{\mathscr{S}}(t_0,s)f_{t_0})(\tilde\boeta^{{j_p}k})\\
=&\frac{1}{d}\sum_{a\in\llbracket1,d\rrbracket}\left(\rU_{\mathscr{S}}(t_0,s){\rm Flat}_a f_{t_0}-{\rm Flat}_a\rU_{\mathscr{S}}(t_0,s) f_{t_0}\right)(\tilde \boeta^{{j_p}k}).
\end{split}
\end{align}
For any $a\in \qq{1,d}$, there are three cases: $\tilde\boeta^{{j_p}k}\not \subset\qq{b_1-a-\psi\ell, b_2+a+\psi\ell}$, $\tilde\boeta^{{j_p}k}\subset\qq{b_1-a+\psi\ell, b_2+a-\psi\ell}$, or neither of them.

For $\tilde\boeta^{{j_p}k}\not \subset\qq{b_1-a-\psi\ell, b_2+a+\psi\ell}$, by our defintion, ${\rm Flat}_a\rU_{\mathscr{S}}(t_0,s) f_{t_0}(\tilde \boeta^{{j_p}k})=1$. Notice that the support of ${\rm{Flat}}_a f_{t_0}-1$ is on $\qq{b_1-a,b_2+a}$. By finite speed of propagation \eqref{finitespeed} in Lemma \ref{l:regpath}, the total mass of 
$\rU_{\mathscr{S}}(t_0,s)({\rm Flat}_a f_{t_0}-1)$ outside $\qq{b_1-a-\psi\ell, b_2+a+\psi\ell}$ is exponentially small. Especially, $|\rU_{\mathscr{S}}(t_0,s){\rm Flat}_a f_{t_0}(\tilde\boeta^{{j_p}k})-1|\leq \exp(-c\psi/2)$.
For $\tilde\boeta^{{j_p}k}\subset\qq{b_1-a+\psi\ell, b_2+a-\psi\ell}$, we have
\begin{align*}
\left|\left(\rU_{\mathscr{S}}(t_0,s){\rm Flat}_a f_{t_0}-{\rm Flat}_a\rU_{\mathscr{S}}(t_0,s) f_{t_0}\right)(\tilde \boeta^{{j_p}k})\right|
=&\left|\left(\rU_{\mathscr{S}}(t_0,s){\rm Flat}_a f_{t_0}-\rU_{\mathscr{S}}(t_0,s) f_{t_0}\right)(\tilde \boeta^{{j_p}k})\right|\\
=&|\left(\rU_{\mathscr{S}}(t_0,s)\left(f_{t_0}-{\rm Flat}_a f_{t_0}\right)\right)(\tilde \boeta^{{j_p}k})|\leq \exp(-c\psi/2),
\end{align*}
we used the finite speed of propagation \eqref{finitespeed} in Lemma \ref{l:regpath} in the last inequality, since $f_{t_0}-{\rm Flat}_a f_{t_0}$ vanishes for any $\boxi$ with $\boxi\in\qq{b_1-a,b_2+a}$.
For the last case, we have $\tilde\boeta^{{j_p}k}\subset\qq{b_1-a-\psi\ell, b_2+a+\psi\ell}$, and some particle of $\tilde\boeta^{{j_p}k}$ is in $\qq{b_1-a-\psi\ell, b_1-a+\psi\ell}\cup \qq{b_2+a-\psi\ell, b_2+a+\psi\ell}$. There are at most $2n\psi \ell$ such $a$, where $n={\cal N}(\tilde\boeta)$ is the total number of particles. Moreover, since $\rU_{\mathscr{S}}$ is a contraction in $L^{\infty}$, we have
%
\begin{align*}
&\left|\left(\rU_{\mathscr{S}}(t_0,s){\rm Flat}_a f_{t_0}-{\rm Flat}_a\rU_{\mathscr{S}}(t_0,s) f_{t_0}\right)(\tilde \boeta^{{j_p}k})\right|\\
\leq &\left|\left(\rU_{\mathscr{S}}(t_0,s){\rm Flat}_a f_{t_0}\right|+\left|{\rm Flat}_a\rU_{\mathscr{S}}(t_0,s){\rm Flat}_{a+2\psi\ell} f_{t_0}\right)(\tilde \boeta^{{j_p},k})\right|+\left|{\rm Flat}_a\rU_{\mathscr{S}}(t_0,s)\left(f_{t_0}-{\rm Flat}_{a+2\psi\ell} f_{t_0}\right)(\tilde \boeta^{{j_p}k})\right|\\
\leq &\|{\rm Flat}_a f_{t_0}\|_{\infty}+\|{\rm Flat}_{a+\psi\ell} f_{t_0}\|_{\infty}+e^{-c\psi/2}.
\end{align*}
Since by \eqref{e:deleig}, uniformly for any $i\in \qq{b_1-a-2\psi \ell, b_1+a+2\psi\ell}$, the eigenvector $\bmu_i(t_0)$ is delocalized in the direction $\bmq$, i.e. $N\langle\bmq, \bmu_i(t_0) \rangle^2 \leq \psi^4$ with overwhelming probability. Thus $\|{\rm Flat}_a f_{t_0}\|_{\infty}, \|{\rm Flat}_{a+2\psi\ell} f_{t_0}\|_{\infty}\lesssim \psi^{4n}$, and
\begin{align*}
\left|\left(\rU_{\mathscr{S}}(t_0,s){\rm Flat}_a f_{t_0}-{\rm Flat}_a\rU_{\mathscr{S}}(t_0,s) f_{t_0}\right)(\tilde \boeta^{{j_p}k})\right|\lesssim \psi^{4n}.
\end{align*}
Combining the above three cases together, it follows that \eqref{eqn:(i)} and therefore $\eqref{eqn:term1}\lesssim_n \psi^{4n+1}\ell/d$.

To bound the term \eqref{eqn:term2}, Since $\tilde{\bm\eta}$ is supported on the interval $\qq{b_1-d-\psi\ell, b_2+d+\psi\ell}$, $\tilde\boeta^{j_pk}$ is supported on $\qq{b_1-d-2\psi\ell, b_2+d+2\psi\ell}$ for any $k$ such that $|k-j_p|\leq \ell$. Therefore, by Lemma \ref{l:Lbound} we have
\begin{align*}
\left|({\rm Av}\rU_{\mathscr{S}}(t_0,s)f_{t_0})(\tilde\boeta^{{j_p}k})-({\rm Av}\rU_{\mathscr{B}}(t_0,s)f_{t_0})(\tilde\boeta^{{j_p}k})\right|
&\leq \left|(\rU_{\mathscr{S}}(t_0,s)f_{t_0}-\rU_{\mathscr{B}}(t_0,s)f_{t_0})(\tilde\boeta^{{j_p}k})\right|\\
&\lesssim_n \psi^{4n}N(s-t_0)/\ell.
\end{align*}
As a consequence, \eqref{eqn:term2}$\lesssim_n \psi^{4n}N(s-t_0)/\ell$.

Finally for \eqref{eqn:term3}, similarly $\tilde\boeta^{j_pk}$ is supported on $\qq{b_1-d-2\psi\ell, b_2+d+2\psi\ell}$,
then by \eqref{e:deleig} uniformly for any $j$ in the support of $\tilde\boeta^{j_pk}$, $N\langle\bmq,\bmu_{j}(s)\rangle^{2} \lesssim \psi^{4}$  with overwhelming  probability. Therefore, $f_s(\tilde\boeta^{j_pk})\lesssim \psi^{4n}$, for any $1\leq p\leq m$. The first part  of \eqref{eqn:term3} is
\begin{align*}
&\frac{1}{N}\Im\sum_{k:0<|{j_p}-k|\leq \ell}\frac{({\rm Av}f_{s})(\tilde\boeta^{{j_p}k})}{z_{j_p}-\la_k}
=\frac{1}{N}\Im\sum_{k:0<|{j_p}-k|\leq \ell}\frac{a_{\tilde\boeta^{{j_p}k}}f_{s}(\tilde\boeta^{{j_p}k})+(1-a_{\tilde\boeta^{{j_p}k}})}{z_{j_p}-\la_k}
\\
=&\frac{1}{N}\Im\sum_{k:0<|{j_p}-k|\leq \ell}\frac{a_{\tilde\boeta}f_{s}(\tilde\boeta^{{j_p}k})+(1-a_{\tilde\boeta})+(a_{\tilde\boeta^{{j_p}k}}-a_{\tilde{\boeta}})f_{s}(\tilde\boeta^{{j_p}k})+(a_{\tilde{\boeta}}-a_{\tilde\boeta^{{j_p}k}})}{z_{j_p}-\la_k}\\
=&\frac{1}{N}\Im\sum_{k:0<|{j_p}-k|\leq \ell}\frac{a_{\tilde\boeta}f_{s}(\tilde\boeta^{{j_p}k})+(1-a_{\tilde\boeta})}{z_{j_p}-\la_k}+O\left(\frac{\ell\psi^{4n}}{d}\right),
\end{align*}
where we used that $|a_{\tilde\boeta^{{j_p}k}}-a_{\tilde{\boeta}}|\leq d(\tilde\boeta, \tilde\boeta^{{j_p}k})/d\leq \ell/d$.

From the proof of Lemma \ref{l:Lbound}, and by our choice $\eta\leq \ell/N$, we have 
\begin{align*}
\frac{1}{N}\sum_{k: |{j_p}-k|> \ell}\frac{\eta f_s(\tilde \boeta^{{j_p}k})}{(\la_{{j_p}}-\la_k)^2+\eta^2}\leq \frac{\psi^{4n}N\eta}{\ell},\quad \frac{1}{N}\sum_{k: |{j_p}-k|> \ell}\frac{\eta }{(\la_{{j_p}}-\la_k)^2+\eta^2}\leq \frac{N\eta}{\ell}.
\end{align*}
\eqref{eqn:term3} can be reduced to upper bound the following expression,
\begin{align}\label{e:expdiff1}
\notag&\frac{1}{N}\sum_{k: 0<|{j_p}-k|\leq \ell}\frac{\eta f_s(\tilde \boeta^{{j_p}k})}{(\la_{{j_p}}-\la_k)^2+\eta^2}-\frac{\eta}{(\lambda_{j_p}-\lambda_k)^2+\eta^2}\\
=&\frac{1}{N}\sum_{k\notin\{j_1,\cdots, j_m\}}\frac{\eta f_s(\tilde \boeta^{{j_p}k})}{(\la_{{j_p}}-\la_k)^2+\eta^2}-\Im m_{{\rm fc},s}(z_{j_p})+O\left(\frac{n\psi^{4n}}{N\eta}+\frac{\psi^{4n}N\eta}{\ell}\right).
\end{align}
Moreover, by definition \eqref{feq} the first sum in the above expression is 
\begin{align}\label{e:expdiff2}
&\bE\left(
\left(
\prod_{1\leq q\leq m,q\neq p}\frac{(N\langle \bmq,\bmu_{j_q}(s)\rangle^2)^{k_q-\delta_{pq}}}{a(2(k_{q}-\delta_{pq}))}\right)
\left(
\sum_{k\not\in\{j_1,\dots,j_m\}}\frac{\eta \langle\bmq,\bmu_{k}(s)\rangle^2}{(\la_{j_p}-\la_{k})^2+\eta^2}
\right)
\, \Big | \, (H_0,\bla)
\right).
\end{align}
Thanks to \eqref{evcontrol}, we have
\begin{align*}
\sum_{k\not\in\{j_1,\dots,j_m\}}\frac{\eta \langle\bmq,\bmu_{k}\rangle^2}{(\la_{j_p}-\la_{k})^2+\eta^2}
=&\Im[\langle \bmq, G(s,z_{j_p})\bmq\rangle]-\frac{1}{N}\sum_{q=1}^{m}\frac{\eta N\langle\bmq,\bmu_{j_m}\rangle^2}{(\la_{j_p}-\la_{j_q})^2+\eta^2}\\
=&\Im[m_{{\rm fc},s}(z_{j_p})]+O_{n}\left(\frac{1}{N^{\fb}}+\frac{\psi^2}{\sqrt{N\eta}}+\frac{\psi^{4}}{N\eta}\right),
\end{align*}
with overwhelming probability. As a result, \eqref{e:expdiff2} is bounded by
\begin{align}\label{e:expdiff3}
\eqref{e:expdiff2}= \Im [m_{{\rm fc},s}(z_{j_p})]f_s(\tilde{\boeta}\backslash {j_p})+O_{n}\left(\frac{\psi^{4n}}{N^{\fb}}+\frac{\psi^{4n}}{\sqrt{N\eta}}\right).
\end{align}
Combining \eqref{e:expdiff3} and \eqref{e:expdiff1}, we have the following estimate for \eqref{eqn:term3},
\begin{align*}
\eqref{eqn:term3}=a_{\tilde\boeta} \Im [m_{{\rm fc},s}(z_{j_p})]\left(f_s(\tilde{\boeta}\backslash {j_p})-1\right)+O_{n}\left(\psi^{4n}\left(\frac{1}{N^{\frak{b}}}+\frac{\ell}{d}+\frac{1}{\sqrt{N\eta}}\right)\right).
\end{align*}
\eqref{eqn:keybound} follows from combining the error estimate of \eqref{eqn:term1}, \eqref{eqn:term2} and \eqref{eqn:term3}.

With the estimate \eqref{eqn:keybound}, we can start proving \eqref{fest2} by induction. We choose the parameters 
\begin{align}\label{e:choice}
\eta=\psi^{8\tilde n}N^{2\frak{d}-1},\quad \ell=\psi^{12\tilde n+1}N^{3\frak{d}}, \quad t_0=t- \psi\eta=t-\psi^{8\tilde n+1}N^{2\frak{d}-1}.
\end{align}
We can take $\frak{c}$ (as in the control parameter $\psi$ \eqref{control}) and $\fd$ small enough such that $\fd+4\tilde n \fc\leq \fb$, $4\fd+(16\tilde n +2)\fc-1\leq \log_N r$ and $2\fd+(8\tilde n+1)\fc-1\leq \log_N(t/2)$, then 
\begin{align*}
&\eta^*\ll t_0\leq t\ll r, \quad t-t_0\leq \frac{\ell}{N}, \quad \frac{\psi^4}{N}\leq \eta\leq \frac{\ell}{N}, \quad t_0+\frac{\ell}{N}\ll r,
\end{align*}
and for any $0\leq n\leq \tilde n$, it holds
\begin{align*}
&\psi^{4n}\left(\frac{1}{N^{\frak{b}}}+\frac{\psi\ell}{d}+\frac{1}{\sqrt{N\eta}}+\frac{N\eta}{\ell}\right) \leq \frac{4}{N^{\fd}},\quad \frac{\psi^{4n}N(t-t_0)}{\ell}\leq \frac{1}{N^{\fd}}.
\end{align*}
Thus \eqref{eqn:keybound} can be simplified as: for any $1\leq n\leq \tilde{n}$, and $t_0\leq s\leq t$, we have 
\begin{align*}
\sum_{k:0<|{j_p}-k|\leq \ell}\Im \frac{g_s(\tilde\boeta^{{j_p}k})}{z_{j_p}-\lambda_k}-\Im \frac{1}{z_{j_p}-\lambda_k}
\lesssim_n
\sup_{\boeta:\mathcal{N}(\boeta)=n-1,\atop 
 \boeta\subset\llbracket b_1+\psi\ell,b_2-\psi\ell\rrbracket}\left|f_s({\boeta})-1\right|+\frac{1}{N^{\frak{d}}} .
\end{align*}
If we plug in this back to \eqref{eqn:tobound}, we have either $g_s(\tilde \boeta)-1\leq N^{-1}$ or
\begin{align}
\label{eqn:diffbound}\del_s \left(g_s(\tilde{\bm \eta})-1\right) 
\lesssim_n&-\frac{1}{\eta}\left(g_s(\tilde{\bm \eta})-1\right)
+\frac{1}{\eta}\left(n\sup_{\boeta:\mathcal{N}(\boeta)=n-1,\atop \boeta\subset\llbracket b_1+\psi\ell,b_2-\psi\ell\rrbracket}\left|f_s(\boeta)-1\right|+\frac{1}{N^{\frak{d}}} \right).
\end{align}

We can prove the following by induction on $n$: let $t_k=t_0+k\psi\eta/\tilde n$ for $k=0,1,2,\cdots, \tilde n$. Then for any time $t_n\leq s\leq t$ we have
\begin{align}\label{ind}
\sup_{\boeta:\mathcal{N}(\boeta)=n,\atop \boeta\subset\llbracket b_1+\psi\ell,b_2-\psi\ell\rrbracket}\left|f_s(\boeta)-1\right|\lesssim_n \frac{1}{N^{\frak{d}}}.
\end{align}
\eqref{ind} holds trivially for $n=0$. Assume \eqref{eqn:diffbound} holds for $n-1$, we prove it for $n$. By induction , for any $t_{n-1}\leq s\leq t$ we have
\begin{align*}
\del_s \left(g_s(\tilde{\bm \eta})-1\right) 
\lesssim_n -\frac{1}{\eta}\left(g_s(\tilde{\bm \eta})-1\right)
+\frac{1}{N^{\frak{d}}\eta}.
\end{align*}
Therefore for any $t_{n}\leq s\leq t$, Gronwall's inequality leads to 
\begin{align*}
\sup_{\boeta:\cal N(\boeta)=n}g_s(\bm \eta)-1\lesssim_n \frac{1}{N^{\frak{d}}} +\left(\sup_{\boeta:\cal N(\boeta)=n}g_{t_{n-1}}(\boeta)-1\right)e^{-\psi}\lesssim_n \frac{1}{N^{\frak{d}}},
\end{align*}
for $N$ large enough. Combining with \eqref{eqn:diffdynamic}, we obtain,
\begin{align*}
\sup_{\boeta:\cal N(\boeta)=n,\atop
\boeta\in\qq{b_1+\psi\ell,b_2-\psi\ell}}f_s(\boeta)-1\lesssim_n \frac{1}{N^{\frak{d}}}.
\end{align*}
Similarly by a minimum principle argument, one can show that 
\begin{align*}
\inf_{\boeta:\cal N(\boeta)=n,\atop
\boeta\in\qq{b_1+\psi\ell,b_2-\psi\ell}}f_s(\boeta)-1\gtrsim_n -\frac{1}{N^{\frak{d}}},
\end{align*}
and \eqref{ind} follows for any $0\leq n\leq \tilde n$.
\end{proof}

\begin{proof}[Proof of Corollary \ref{c:que}]
By taking $\bmq$ supported on $i$ and $j$-th coordinates in Theorem \ref{t:normal}, we know that for any $k$ such that $\lambda_k(t)\in I_{2\kappa}^r(E_0)$, $u_{ki}^2(t)$ and $u_{kj}^2(t)$ are jointly asymptotically normal.
A second moment calculation yields
\begin{align*}
&\bE\left[\left(\frac{N}{\|\bf a\|_1}\sum_{i=1}^N a_i u_{ki}^2(t)\right)^2\right]=
\frac{1}{\|{\bf a}\|^2_1}\bE\left[\left(\sum_{i=1}^N a_i(N u_{ki}^2(t)-1)\right)^2\right]\\
\leq& \max_{i\neq j}\left|\bE\left[\left(N u_{ki}^2(t)-1\right)\left(N u_{kj}^2(t)-1\right)\right]\right|
+\frac{1}{\|{\bf a}\|_1}\max_i\bE\left[\left(N u_{ki}^2(t)-1\right)^2\right].
\end{align*}
By Theorem \ref{t:normal}, the first term of the right hand side  is bounded by  $C N^{-\frak{d}}$, 
and the second term is bounded by $C/\|{\bf a}\|_1$, where $C$ is an universal constant.
The Markov inequality then allows us to conclude the proof of  \eqref{localque}. 
\end{proof}

\section{Proof of Theorem \ref{t:evsparse}}\label{s:app}

For the proof of Theorem \ref{t:evsparse}, we follow the three-step strategy as in \cite{MR3429490, Univregular}, where it was proved for sparse Erd\H{o}s-R\'{e}nyi graphs in the regime $N^{\delta}\leq p\leq N/2$ in \cite{MR3429490}, and $p$-regular graphs in the regime $N^\delta\leq p\leq N^{2/3-\delta}$ in \cite{Univregular},  that in the bulk of the spectrum the local eigenvalue correlation functions and the distribution of the gaps between consecutive eigenvalues coincide with those of the Gaussian orthogonal ensemble. We prove Theorem \ref{t:evsparse} for  Erd\H{o}s-R\'{e}nyi graphs, the proof for $p$-regular graphs is similar, and we only remark the differences.

Before the proof of Theorem \ref{t:evsparse}, we recall some definitions and notations from \cite{MR3429490}. 

\begin{definition}\label{general}
Let $A$ be an $N\times N$ deterministic real symmetric matrix. We denote the eigenvalues of $A$ as $\lambda_1\leq \lambda_2\leq\cdots\leq \lambda_N$ and corresponding eigenvectors $\bmu_1,\bmu_2,\cdots, \bmu_N$. For any small parameter $\fc>0$ and unit vector $\bmq\in \bR^N$, we call the matrix $A$ $(\fc,\bmq)$-general, if there exists an universal constant $C$ such that 
\begin{enumerate}
\item The eigenvectors of $A$ are delocalized in all base directions and direction $\bmq$: for all $i,j\in \qq{N}$, $\langle \bme_j, \bmu_i\rangle^2, \langle \bmq, \bmu_i\rangle^2\leq CN^{-1+\fc}$.
\item The eigenvalues of $A$ do not accumulate:  there is an universal constant $C$, such that for any interval $I$ with length $|I|\geq N^{-1+\fc}$, we have $\#\{i: \lambda_i\in I\}\leq C|I|N$. 
\end{enumerate} 
\end{definition}

We recall the quantity $Q_i$ on the space of symmetric $N\times N$ real matrices. For any $N\times N$ matrix $A$, if $\lambda_i$ is a single eigenvalue of $A$, $Q_i(A)$ is given by 
\begin{align}\label{Q}
Q_i(A)=\frac{1}{N^2}\sum_{j:j\neq i}\frac{1}{|\lambda_j-\lambda_i|^2}.
\end{align}
This quantity plays an important role in \cite{MR2784665, MR2669449}, where it was observed that $Q_i(A)$ captures quantitatively the derivatives of the eigenvalues $\lambda_i$ of $A$. 

\begin{proposition}\label{bdrQ}
Let $A$ be an $N\times N$ deterministic real symmetric matrix. If $A$ is $(\fc, \bmq)$-general in the sense of Definition \ref{general} and 
\begin{align}
\label{gapb}Q_i(A)=\frac{1}{N^2}\sum_{j:j\neq i}\frac{1}{|\lambda_{i}(A)-\lambda_{j}(A)|^2}\leq M=N^{2\tau},
\end{align} then there exists some universal constant $C$ such that
\begin{align}
\label{drei}&|\del_{ab}^{(k)}\lambda_i(A)|\leq CN^{-1+(k-1)\tau+(2k-1)\fc}, \quad k=1,2,3,\\
\label{drQ}&|\del_{ab}^{(k)}Q_i(A)|\leq CN^{(k+2)\tau+(2k+2)\fc},\quad k=1,2,3,\\
\label{drev}&|\del_{ab}^{(k)}\langle\bmq, \bmu_i(A) \rangle^2|\leq CN^{-1+k\tau+(2k+1)\fc},\quad k=1,2,3,
\end{align}
where $\del_{ab}$ is the derivative with respective to $(a,b)$-th entry of $A$.
\end{proposition}
\begin{proof}
The first two estimates \eqref{drei} and \eqref{drQ} are proved in \cite[Proposition 4.6]{MR3429490}. The proof of \eqref{drev} is analogous. We denote $G=(A-z)^{-1}$ the resolvent of $A$, and $V$ the matrix whose matrix elements are zero everywhere except at the $(a,b)$ and $(b,a)$ position, where it equals one. For the derivative of eigenvectors \eqref{drev}, we use the following contour integral formula:
\begin{align*}
\del_{ab}^{(k)}\langle \bmq, u_i\rangle^2
=\del_{ab}^{(k)}\oint \langle \bmq, G(z) \bmq \rangle dz
=(-1)^k k! \oint \langle \bmq, (G(z)V)^kG \bmq \rangle dz,
\end{align*}
where the contour encloses only $\lambda_i$. \eqref{drev} follows from analogue estimate as in the proof of \cite[Proposition 4.6]{MR3429490}. For example 
\begin{align*}
\left|\del_{ab}\langle \bmq, u_i\rangle^2\right|
=\left|2\sum_{j:j\neq i}\frac{\langle \bmq, u_i\rangle\langle \bmq, u_j\rangle (u_i^*Vu_j)}{\lambda_j-\lambda_i}\right|
\leq \frac{2}{N^{2-2\fc}}\sum\left|\sum_{j:j\neq i}\frac{1}{|\lambda_j-\lambda_i|}\right| \leq CN^{-1+\tau+3\fc},
\end{align*}
thanks to the delocalization of eigenvectors of $A$ in directions $\bme_a$, $\bme_b$ and $\bmq$. 
\end{proof}

Recall that $H$ is the normalized adjacency matrix of Erd\H{o}s-R\'{e}nyi graphs as given in section \ref{s:intro}. We define the following matrix stochastic differential equation which is an Ornstein-Uhlenbeck version of the Dyson Brownian motion. The dynamics of the matrix entries are given by the stochastic differential equations
\begin{align}
\label{OUmod0} \rd\left(h_{ij}(t)-\f\right) =\frac{\rd w_{ij}(t)}{\sqrt{N}}-\frac{1}{2}\left(h_{ij}(t)-\f \right)\rd t, \quad f=\frac{p/N}{\sqrt{p(1-p/N)}}.
\end{align}
where $W_t=(w_{ij}(t))_{1\leq i\leq j\leq N}$ is symmetric with $(w_{ij}(t))_{1\leq i\leq j\leq N}$ a family of independent Brownian motions of variance $(1+\delta_{ij})t$. We denote $H_t= (h_{ij}(t))_{1\leq i,j\leq N}$, and so $H_0=H$ is our original matrix. More explicitly, for the entries of $H_t$, we have
\begin{align}
\label{fDBM} h_{ij}(t)=\f+e^{-\frac{t}{2}}\left(h_{ij}(0)-\f\right)+\frac{1}{\sqrt{N}}\int_{0}^{t} e^{\frac{s-t}{2}}\rd{w_{ij}(s)}.
\end{align}
Clearly, for any $t\geq 0$ and $i<j$, we have $\bE[h_{ij}(t)]=\f$, and $\bE[\left(h_{ij}(t)-\f\right)^2]=1/N$. More importantly, the law of $h_{ij}(t)$ is Gaussian divisible, i.e. it contains a copy of Gaussian random variable with variance $O(tN^{-1})$. Therefore $H_t$ can be written as 
\begin{align}
\label{modOU}H_t\stackrel{d}{=}\tilde H_t+\sqrt{1-e^{-t}}G, \quad \tilde H_t =\f +e^{-t/2}(H-\f),
\end{align}
where $G$ is a standard Gaussian orthogonal ensemble, i.e., $G=(g_{ij})_{1\leq i\leq j\leq N}$ is symmetric with $(g_{ij})_{1\leq i\leq j\leq N}$ a family of independent Brownian motions of variance $(1+\delta_{ij})/N$, and is independent of $\tilde H_t$.

\begin{proposition}\label{p:initialest}
For $N^{\delta}\leq p\leq N/2$, we fix $0<\fb\leq \delta/3$. Then for $0\leq s\ll 1$, any unit vector $\bmq\in \bR^N$ such that $\bmq \perp \bme$ (where $\bme=(1,1,\cdots, 1)^*/\sqrt{N}$) and $N$ large enough, the followings hold.
\begin{enumerate}
\item For any $\fc>0$, with overwhelming probability $H_s$ is $(\fc,\bmq)$-general in the sense of definition \ref{general}.
\item Assumptions \ref{a:boundImm} and \ref{a:boundEv} hold for $\tilde H_s$ (as in \eqref{modOU}), with overwhelming probability. More precisely, $\|\tilde H_s\|\leq N$, and uniformly for any $z\in \{E+\i \eta: |E|\leq 5, N^{3\fb-1}\leq \eta\leq 1\}$
\begin{align}\label{e:sparseest}
|\Tr (\tilde H_s-z)^{-1}/N-m_{\rm{sc}}(z)|\leq N^{-\fb},\quad |\langle \bmq, (\tilde H_s-z)^{-1}\bmq\rangle-m_{\rm{sc}}(z)|\leq N^{-\fb},
\end{align}
with overwhelming probability. 
\end{enumerate}
%
\end{proposition}
\begin{proof}
For any $0\leq s\ll 1$, $H_s, \tilde H_s$ belong to the family of sparse random matrices in \cite{MR3098073} with sparsity $\sqrt{p}$. Under our normalization, with overwhelming probability $\|\tilde H_s\|\leq C\sqrt{p}\ll N$. We denote $G(z)=(H_s-z)^{-1}$ (or $(\tilde H_s-z)^{-1}$) and $m(z)
=\Tr G(z)/N$. Thanks to \cite[Theorem 2.8]{MR3098073}, with overwhelming probability, 
\begin{align}\label{locallaw}
|m(z)-m_{\rm{sc}}(z)|\leq \max_{i,j\in \qq{N}}|G_{ij}(z)-\delta_{ij}\msc(z)|\leq (\log N)^{C\log \log N}\left(\frac{1}{p^{1/2}}+\frac{1}{(N\eta)^{1/2}}\right)
\end{align}
uniformly for any $z\in \{E+\i \eta: |E|\leq 5, 0< \eta\leq 1\}$, where $\msc$ is the Stieltjes transform of the semi-circle distribution.  More, noticing that $H_s, \tilde H_s$ are exchangeable random matrices, it follows from the entry-wise local law \eqref{locallaw} and \cite[Theorem 8.2]{Rigidregular},
\begin{align}\label{isolocallaw}
|\langle \bmq, G(z)\bmq\rangle-\msc(z)|\leq (\log N)^{C\log \log N}\left(\frac{1}{p^{1/2}}+\frac{1}{(N\eta)^{1/2}}\right),
\end{align}
with overwhelming probability. It follows that $H_s$ is $(\fc, \bmq)$-general, and \eqref{e:sparseest} holds for $\tilde H_s$, with overwhelming probability, and thus Assumption \ref{a:boundEv} holds for $\tilde{H}_s$. For Assumption \eqref{a:boundImm}, fix any $\kappa>0$, since on $z\in\{E+\i \eta: E\in [-2+\kappa, 2-\kappa], 0<\eta\leq 1\}$, there exists a constant $C$ such that $2C^{-1}\leq \Im[\msc(z)]\leq C/2$. Therefore it follows from \eqref{e:sparseest} that $C^{-1}\leq \Im[m(z)]\leq C$ on $z\in\{E+\i \eta: E\in [-2+\kappa, 2-\kappa], N^{3\fb-1}<\eta\leq 1\}$.


\end{proof}

\begin{remark}
We believe the technical assumption $\bmq \perp \bme$ is not necessary, the isotropic local law \eqref{isolocallaw} holds for any unit vector $\bmq\in \bR^N$. However, the proof in \cite[Theorem 8.2]{Rigidregular} works only for unit vectors perpendicular to $\bme$.
\end{remark}

Thanks to Proposition \ref{p:initialest}, with overwhelming probability, $\tilde H$ satisfies Assumptions \ref{a:boundImm} and \ref{a:boundEv}. In \eqref{modOU}, if we condition on those good initial data $\tilde H$, the eigenvectors of $H_t$ are asymptotically normal with overwhelming probability with respect to the randomness of $G$ (as in \eqref{modOU}). If we then take expectation with respect to $\tilde H$, the following proposition follows.

\begin{proposition}\label{p:smallG}
Fix $\kappa>0$, $0<\fb\leq\delta/3$, positive integer $n>0$ and polynomial $P$ of $n$ variables. Then for any $N^{4\fb-1}\leq t\ll 1$, unit vector $\bmq\in \bR^N$ perpendicular to $\bme$ (where $\bme=(1,1,\cdots, 1)^*/\sqrt{N}$),  indexes $i_1, i_2,\cdots, i_n \in \qq{\kappa N, (1-\kappa)N}$ and $N$ large enough, there exists a constant $\fd$ depending on $\fb, t$,
\begin{equation}
\left|\bE\left[P\left(\left(N\langle \bmq,\bmu_{i_k} (t)\rangle^2\right)_{1\leq k\leq n}\right) \right]-\bE\left[ P\left((|\mathscr{N}_i|^2)_{i=1}^n\right)\right]\right|\leq CN^{-\fd},
\end{equation}
where  $\bmu_i(t)$ are eigenvectors of $H_t$ corresponding to $i$-th eigenvalue, and $\mathscr N_i$ are independent standard normal random variables.
\end{proposition}

\begin{proof}[Proof of Theorem \ref{t:evsparse}]
For simplicity of notation, we only state the proof for $n=1$ case, i.e. we fix time $t=N^{4\fb-1}$, and prove that for any $i\in \qq{\kappa N, (1-\kappa)N}$
\begin{align}\label{oneg}
\left|\bE[P(N \langle \bmq, \bmu_i(0)\rangle^2)]-\bE[P(N\langle \bmq, \bmu_i(t)\rangle^2)]\right|\leq CN^{-\fd}.
\end{align}

Take a cutoff function $\rho_M$ such that $\rho_M(x)=1$ for $x\leq M$ and $\rho_M(x)=0$ for $x\geq 2M$, where $M=N^{2\tau}$ and $\tau>0$ is a small constant. By the level repulsion of $H$ and $H_t$ from \cite[Theorem 4.1]{MR3429490}, we know that
\begin{align*}
\bP(Q_i(H_s)\geq N^{2\tau})\leq N^{-\tau/2}, \quad s=0, t.
\end{align*} 
Let $m$ be the degree of $P$, we have that $P(x)\leq Cx^m$. By \eqref{p:initialest}, $H_s$ is $(\fc,\bmq)$-general, especially, with overwhelming probability $N\langle \bmq, \bmu_i(s)\rangle^2\leq CN^{\fc}$, for $s=0,t$. Therefore 
$\bE[P^2(N\langle \bmq, \bmu_i(s)\rangle^2)]\leq CN^{2m\fc}$, and we have
\begin{align*}
&\left|\bE[P(N \langle \bmq, \bmu_i(0)\rangle^2)]-\bE[P(N\langle \bmq, \bmu_i(t)\rangle^2)]\right|\\
\leq &\left|\bE[P( N
 \langle \bmq, \bmu_i(0)\rangle^2)\rho_{M}(Q_i(H_0))]-\bE[P( N\langle \bmq, \bmu_i(t)\rangle^2)\rho_{M}(Q_i(H_t))]\right|\\
+&\bE[P^2( N\langle \bmq, \bmu_i(0)\rangle^2)]^{1/2}\bP(Q_i(H_0)\geq N^{2\tau})+\bE[P^2( N\langle \bmq, \bmu_i(t)\rangle^2)]^{1/2}\bP(Q_i(H_t)\geq N^{2\tau})\\
\leq &\left|\bE[O(N \langle \bmq, \bmu_i(0)\rangle^2)\rho_{M}(Q_i(H_0))]-\bE[P(N\langle \bmq, \bmu_i(t)\rangle^2)\rho_{M}(Q_i(H_t))]\right|+CN^{-\tau/2+m\fc}.
\end{align*}
Notice that $P(N\langle q, \bmu_i(A)\rangle^2)\rho_{M}(Q_i(A))$ is a well defined smooth function on the space of symmetric functions. Moreover, if the matrix $A$ is $(\fc, \bmq)$-general in the sense of Definition \ref{general}, the eigenvectors of $A$ are delocalized,  and $Q_i(A)\leq M=N^{2\tau}$, 
then from Proposition \ref{bdrQ}, we have 
\begin{align*}
\left|\del_{ab}^{(3)}P(N\langle\bmq, \bmu_i(A)\rangle^2)\rho_{M}(Q_i(A))\right|\leq CN^{(m+8)\fc+5\tau},
\end{align*}
where $C$ is a  universal constant. Therefore by \cite[Lemma 4.3]{MR3429490}, we have
\begin{align*}
\left|\bE[P(N \langle \bmq, \bmu_i(0)\rangle^2)]-\bE[P(N\langle \bmq, \bmu_i(t)\rangle^2)]\right|\leq CtN^{1+(m+8)\fc+5\tau}p^{-1/2}\leq CN^{-\fd},
\end{align*}
provided we take $t=N^{4\fb-1}$, and $\fb, \fc,\tau$ small enough such that $4\fb+(m+8)\fc+5\tau-\delta/2\leq \fd$ and $\fb\leq \delta/3$. Combining with Proposition \ref{p:smallG}, it follows
\begin{align*}
\left|\bE[P(N \langle \bmq, \bmu_i(0)\rangle^2)]-\bE[P(\mathscr N^2)]\right|\leq  CN^{-\fd},
\end{align*}
where $\mathscr N$ is a standard normal random variable.

The proof for the case of $p$-regular graphs is analogous. Let $H$ be the normalized adjacency matrix of $p$-regular graphs as in Section \ref{s:intro}. The isotropic local law of $H$ was proved in \cite{Rigidregular}. Since the adjacency matrix $A$ of a $p$-regular graph is subject to the hard constraints that its rows and columns have sum $p$ (i.e. it has the eigenvector $\bme=(1,1,\cdots, 1)^*/\sqrt{N}$). Therefore, instead of using the usual Dyson Brownian motion \eqref{modOU} as in the Erd\H{o}s-R\'{e}nyi graph case, we use the \emph{constrained} Dyson Brownian motion (as introduced in \cite[Definition 2.2]{Rigidregular}), which is the Dyson Brownian motion constrained to the subspace of symmetric matrices whose row and column sums vanish. 
Let $H_t$ be the constrained Dyson Brownian motion after time $t$ with initial data $H_0=H$. We denote its   eigenvalues $\lambda_1(t)\leq  \lambda_2(t)\leq \cdots\leq \lambda_{N-1}(t)\leq \lambda_N(t)=p/\sqrt{p-1}$, with corresponding eigenvectors $\bmu_1(t),\bmu_2(t),\cdots(t), \bmu_{N-1}(t),\bmu_N(t)=\bme$.  
Up to a change of basis, the constrained Dyson Brownian motion is equivalent to the usual $(N-1)$-dimensional Dyson Brownian motion normalized by $N$ rather than by $N-1$. More concretely, let $P$ be an isomorphism from $\bme^{\perp}$ to $\bR^{N-1}$, e.g., 
we can take
\begin{align*}
P_{ij}=\delta_{ij} -\frac{1}{\sqrt{N}-1}\left(\frac{1}{\sqrt{N}}-\delta_{jN}\right),\quad i\in\qq{1,N-1}, j\in \qq{1,N}.
\end{align*}
Once we identify $\bme^{\perp}$ with $\bR^{N-1}$ using $P$, the constrained Dyson Brownian motion is the same as the usual $N-1$-dimensional Dyson Brownian motion:
\begin{align*}
P H_t P^* \stackrel{d}= e^{-t/2}PH_0 P^*+\sqrt{1-e^{-t}}G,
\end{align*}
where $G=(g_{ij})_{1\leq i\leq j\leq N-1}$ is symmetric with $(g_{ij})_{1\leq i\leq j\leq N-1}$ a family of independent Brownian motions of variance $(1+\delta_{ij})/N$. Since $\bmu_i(t)\perp \bme$, $P\bmu_i(t)$ for $i\in\qq{N-1}$ are eigenvectors of $PH_tP^*$. Thus $P\bmu_i(t)$ for $i\in\qq{N-1}$ have the same distribution as the eigenvectors of $e^{-t/2}PH_0 P^*+\sqrt{1-e^{-t}}G$. Thanks to Theorem \ref{t:normal}, for $t\gg 1/N$, the bulk eigenvectors of $e^{-t/2}PH_0 P^*+\sqrt{1-e^{-t}}G$ are asymptotically normal in the direction $P\bmq$, which is a unit vector in $\bR^{N-1}$ since $\bmq \perp \bme$. Noticing that $\langle P\bmq, P\bmu_i(t)\rangle =\langle \bmq, \bmu_i(t)\rangle$, we conclude that the bulk eigenvectors of $H_t$ are asymptotically normal in the direction $\bmq$. The same argument as in the proof of Erd\H{o}s-R\'{e}nyi case, combining with the continuity Proposition \cite[Proposition 3.1]{Univregular}, implies that the law of $\{\langle \bmq, \bmu_i(0)\rangle\}_{i=i_1,i_2,\cdots, i_n}$ is asymptotically the same as that of $\{\langle \bmq, \bmu_i(t)\rangle\}_{i=i_1,i_2,\cdots, i_n}$. And thus, the claim of Theorem \ref{t:evsparse} follows.
%
%
%
\end{proof}

%

\bibliography{References}{}
\bibliographystyle{plain}

\end{document}